\documentclass[11pt, a4paper]{article}

\usepackage{amsmath,amssymb,amsbsy,amsfonts,amsthm,latexsym,amsopn,amstext,amsxtra,epic, euscript,amscd,indentfirst,verbatim,graphicx, hyperref,xcolor, setspace,tikz,tikz-cd}
\usepackage{authblk}

\usepackage{tabularx} 
\usepackage{relsize} 
\usepackage[margin=1in]{geometry}
\usepackage{enumitem} 
\usepackage{hyperref} 
\begin{document}
\newcommand{\mres}{\mathbin{\vrule height 1.2ex depth 0pt width
0.13ex\vrule height 0.13ex depth 0pt width 0.9ex}}

\newtheorem{theorem}{Theorem}[section]
\newtheorem*{theorem*}{Theorem}
\newtheorem{conjecture}[theorem]{Conjecture}
\newtheorem*{conjecture*}{Conjecture}
\newtheorem{corollary}[theorem]{Corollary}

\newtheorem{proposition}[theorem]{Proposition}
\newtheorem*{proposition*}{Proposition}
\newtheorem{question}{Question}
\newtheorem{lemma}[theorem]{Lemma}
\newtheorem*{lemma*}{Lemma}
\newtheorem{cor}[theorem]{Corollary}
\newtheorem*{obs*}{Observation}
\newtheorem{obs}{Observation}
\newtheorem{condition}{Condition}
\newtheorem{definition}{Definition}
\newtheorem{remark}{Remark}
\newtheorem*{definition*}{Definition}
\newtheorem{proc}[theorem]{Procedure}
\newtheorem{problem}{Problem}
\newcommand{\comments}[1]{} 

%
\newcommand{\bF}{\mbf{F}}

\newcommand{\bd}{\mbf{d}}
\newcommand{\bw}{\mbf{w}}
\newcommand{\bbX}{\mathbb{X}}
\newcommand{\bY}{\mbf{Y}} 
\newcommand{\bby}{\mathbb{Y}}
\newcommand{\bbY}{\mathbb{Y}}
\newcommand{\bbV}{\mathbb{V}}
\newcommand{\balpha}{\mbf{\alpha}}
\newcommand{\bgamma}{\mbf{\gamma}}
\newcommand{\bL}{\mbf{L}}
\newcommand{\bZ}{\mbf{Z}}
\newcommand{\bbZ}{\mathbb{Z}}
\newcommand{\bXi}{\mbf{\Xi}}
\newcommand{\bW}{\mbf{W}}
\newcommand{\bV}{\mbf{V}}
\newcommand{\mI}{\mathcal{I}}
\newcommand{\mK}{{K}}
\newcommand{\mR}{\mathcal{R}}

\newcommand{\mL}{\mathcal{L}}
\newcommand{\mH}{\mathcal{H}_{{K}}}

\newcommand{\rhoT}{\rho_T}
\newcommand{\rhoL}{\rho_T^L}

\newcommand{\intkernele}{{\intkernel^{}}}
\newcommand{\intkernela}{{\intkernel^{A}}}
\newcommand{\intkernelxi}{\intkernel^{\xi}}
\newcommand{\bintkernela}{\bintkernel^{A}}
\newcommand{\bintkernelxi}{\bintkernel^{\xi}}
\newcommand{\bintkernele}{\bintkernel^{E}}
\newcommand{\intkernelvare}{\varphi^{E}}
\newcommand{\intkernelvara}{\varphi^{A}}
\newcommand{\intkernelvarxi}{\varphi^{\xi}}
\newcommand{\bintkernelvare}{\bintkernelvar^{E}}
\newcommand{\bintkernelvara}{\bintkernelvar^{A}}
\newcommand{\bintkernelvarxi}{\bintkernelvar^{\xi}}

\newcommand{\basis}{\psi}
\newcommand{\basise}{\psi^{\bx}}
\newcommand{\basisa}{\psi^{\dot\bx}}
\newcommand{\basisxi}{\psi^{\xi}}

\newcommand{\Tmixing}{T_{\mathrm{mix}}}
\newcommand{\eqnarrayterm}{\Delta}
\newcommand{\force}{\mbf{F}}
\newcommand{\forcex}{\mathbf{F}^{\bx}}
\newcommand{\forcev}{\force^{\bv}}
\newcommand{\forcexi}{\force^{\xi}}
\newcommand{\intkernel}{\phi}
\newcommand{\lintkernel}{\widehat{\intkernel}}
\newcommand{\bintkernel}{{\bm{\phi}}}
\newcommand{\blintkernel}{{\widehat{\bm{\phi}}}}
\newcommand{\intkernelvar}{\varphi}
\newcommand{\bintkernelvar}{{\bm{\varphi}}}
\newcommand{\bintkerneltrue}{{\bm{\phi}}_{true}}
\newcommand{\bintkernelv}{{\bm{\intkernel}}}
\newcommand{\rhsfo}{\mathbf{f}}
\newcommand{\hypspace}{\mathcal{H}}
\newcommand{\bhypspace}{\mbf{\mathcal{H}}}
\newcommand{\bhypspaceEA}{\mbf{\mathcal{H}}^{EA}}
\newcommand{\bhypspacexi}{\mbf{\mathcal{H}}^{\xi}}

\newcommand{\phiH}{\intkernel_{\mathcal{H}}}
\newcommand{\thetae}{{\theta^E}}
\newcommand{\thetaa}{{\theta^A}}

\newcommand{\bX}{\boldsymbol{X}}
\newcommand{\bx}{\boldsymbol{x}}
\newcommand{\br}{\boldsymbol{r}}
\newcommand{\by}{\boldsymbol{y}}
\newcommand{\bz}{\boldsymbol{z}}

\newcommand{\E}{\mathbb{E}}
\newcommand{\cov}{\mathrm{Cov}}
\newcommand{\state}{\boldsymbol{X}}
\newcommand{\traj}{\state_{[0,T]}}
\newcommand{\probIC}{\mu_0}
\newcommand{\intkerneltrue}{\phi_{true}}
\newcommand{\ml}{\text{MATLAB}^{\texttrademark}}
\newcommand{\apr}{\textit{ a priori }}

\newcommand{\norm}[1]{\left\| #1 \right\|}
\newcommand{\infnorm}[1]{\| #1\|_{\infty}}
\newcommand{\hnorm}[1]{\| #1\|_{\mathcal{H}}}
\newcommand{\Rhoxnorm}[1]{\| #1\|_{L^2(\rho_{\mbf{X}})}}
\newcommand{\rhotnorm}[1]{\| #1\|_{L^2(\tilde\rho_{T}^L)}}

\newcommand{\hinnerp}[2]{\langle #1, #2\rangle_{\mH}}
\newcommand{\Rhoxinnerp}[2]{\langle #1, #2\rangle_{L^2(\rho_{\mbf{X}})}}
\newcommand{\rhotinnerp}[2]{\langle #1, #2\rangle_{L^2(\tidle\rho_{T}^L)}}

\newcommand{\supp}[1]{\text{supp}(#1)}

\newcommand{\realR}[1]{\mathbb{R}^{#1}}
\newcommand{\real}{\mathbb{R}}
\newcommand{\R}{\real}

\newcommand{\argmin}[1]{\underset{#1}{\operatorname{arg}\operatorname{min}}\;}
\newcommand{\arginf}[1]{\underset{#1}{\operatorname{arg}\operatorname{inf}}\;}
\newcommand{\argmax}[1]{\underset{#1}{\operatorname{arg}\operatorname{max}}\;}

\def\Z{\mathbb Z}
\def\Za{\mathbb Z^\ast}
\def\Fq{{\mathbb F}_q}
\def\R{\mathbb R}
\def\N{\mathbb N}
\def\C{\mathbb C}
\def\k{\kappa}
\def\grad{\nabla}
\def\M{\mathcal{M}}
\def\S{\mathcal{S}}
\def\pt{\partial}
\newcommand{\todo}[1]{\textbf{\textcolor{red}{[To Do: #1]}}}
\newcommand{\note}[1]{\textbf{\textcolor{blue}{#1}} \\ \\}
\title{ On the Identifiability of  Nonlocal Interaction Kernels in First-Order Systems of Interacting Particles on Riemannian Manifolds}

\author[1]{Sui Tang \thanks{Email: suitang@ucsb.edu}}
\author[1]{Malik Tuerkoen\thanks{Email: 
malik@math.ucsb.edu}}
\author[1]{ Hanming Zhou \thanks{ Email: hzhou@math.ucsb.edu}}
\affil[1]{Department of Mathematics, University of California Santa Barbara}

\maketitle
\begin{abstract}

In this paper, we tackle a critical issue in nonparametric inference for systems of interacting particles on Riemannian manifolds: the identifiability of the interaction functions. Specifically, we define the function spaces on which the interaction kernels can be identified given infinite i.i.d observational derivative data sampled from a distribution. Our methodology involves casting the learning problem as a linear statistical inverse problem using an operator theoretical framework. We prove the well-posedness of the inverse problem  by establishing the strict positivity of a related integral operator and our analysis allows us to refine the results on specific manifolds such as the sphere and Hyperbolic space.  Our findings indicate that a numerically stable procedure exists to recover the interaction kernel from finite (noisy) data, and the estimator will be convergent to the ground truth. 
 These findings also answer an open question in \cite{Maggioni2021LearningIK} and demonstrate that least square estimators can be statistically optimal in certain scenarios. Finally, our theoretical analysis  could be extended to the mean-field case, revealing that the corresponding nonparametric inverse problem is ill-posed in general and necessitates effective regularization techniques.

\end{abstract}

\section{Introduction} 

Systems of interacting particles are ubiquitous in science and engineering, where the particles may refer to fundamental particles in physics, planets in astronomy, animals in ecology, and cells in biology. These systems exhibit a wide range of collective behaviors at different scales and levels of complexity arising from individual interactions among particles. Understanding and simulating such collective behavior requires the development of effective differential equation models, which is a central subject in applied mathematics.

In many applications, particles may be associated with state variables that are defined on or constrained to move in non-Euclidean spaces. This presents significant difficulties for modeling, analysis, and numerical simulation. Despite these challenges, there has been a growing interest in the last five years in modeling particles moving on various manifolds or surfaces. One such example is a simple first-order system that models the consensus behavior of opinion dynamics on a general Riemannian manifold. This system considers $N$ interacting particles on a Riemannian manifold $(\mathcal M,g)$,  which evolve according to a general dynamics equation shown in (\ref{system}):

\begin{align}\label{system}
\dot \bx_i=\frac1N\sum_{j=1}^N\phi(d(\bx_i,\bx_j))w(\bx_i,\bx_j), i=1,\cdots,N. 
\end{align}
Here, $\bx_i$ represents the position of the $i$-th particle on the manifold $\mathcal{M}$. The radial interaction kernel $\phi$ is a scalar valued function defined over $\mathbb{R}^+$; the distance function $d(\cdot, \cdot)$ with respect to the Riemannian metric $g$, and influence vector $w(\bx_i,\bx_j)\in T_{\bx_i}\mathcal{M}$ (the tangent space at $\bx_i$) are all defined in the equation. Other models, such as flocking models in \cite{ahn2021emergent,ahn2022emergent}, aggregation models in \cite{fetecau2018self,fetecau2021well}, and Kuramoto models in \cite{strogatz2000kuramoto,ha2022emergent}, have also been shown to reproduce various qualitative patterns of collective dynamics.

In the field of modelling, one of the greatest challenges is selecting the appropriate governing equations to accurately describe the desired collective behaviour. In the past, this task relied heavily on the expertise of the modeler. However, modern sensor and measurement technologies now allow for the collection of large amounts of high-quality data from a diverse range of systems. As a result, the discovery of interacting particle models that accurately match observational data has become a major area of focus in recent years. This task is often difficult and suffers from the curse of dimensionality, but many systems have low-dimensional structures that enable efficient data-driven methods. Examples of this include the inference of stochastic interacting particle systems in works such as \cite{kasonga1990maximum, bishwal2011estimation, gomes2019parameter, chen2021maximum, sharrock2021parameter, messenger2021learning, genon2022inference, della2022lan, yao2022mean}, and the learning of radial interaction kernels in \cite{lu2021learning}. Deterministic interacting particle systems on Euclidean domains have also been studied, with works such as \cite{bongini2017inferring, lu2019nonparametric, lu2021learning, miller2020learning} focusing on  nonparametric inference methods.

\begin{figure}[htbp]
    \centering
    \includegraphics[scale=0.4]{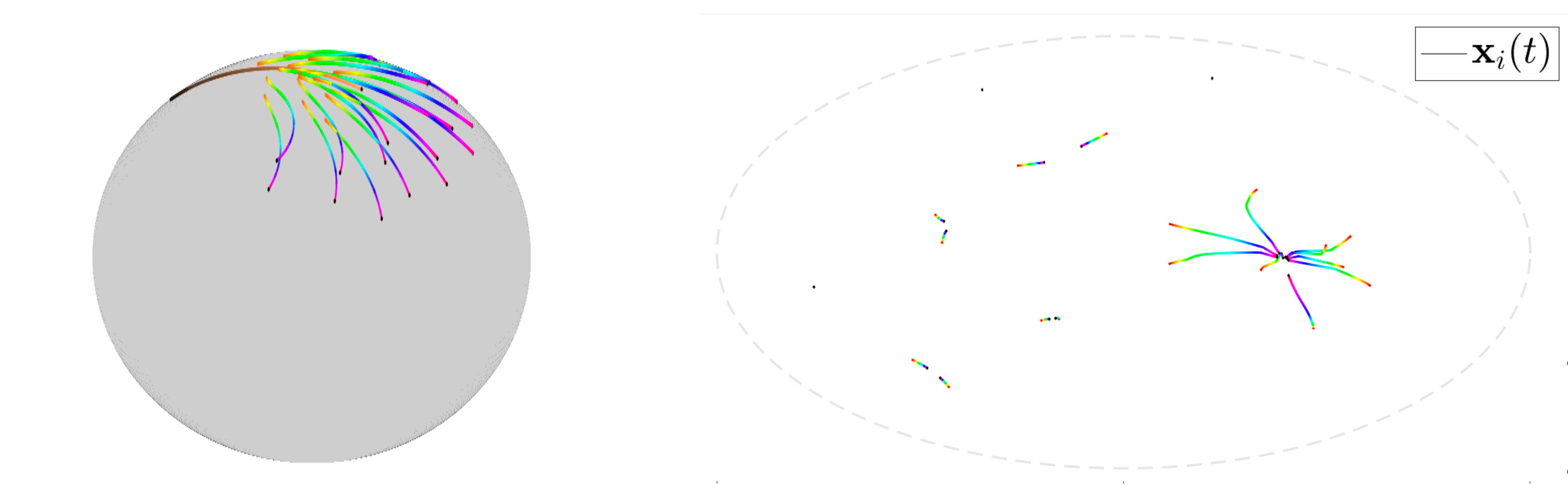}
    \caption{The two examples of systems \eqref{system} using code provided in \cite{Maggioni2021LearningIK}. The left models the swarming behaviour of preys and predators on $\mathbb{S}^2$ and the right models the aggregation behaviours of opinions on Poincare disk. \textcolor{black}{The color variations from orange to pink denotes the forward time evolution of particles.} }
    \label{fig:my_label}
\end{figure}

In their work \cite{Maggioni2021LearningIK}, the authors propose a least squares algorithm that uses the $\ell^2$ ODE residual \eqref{system} as the loss function and looks for an approximation of $\phi$ over piecewise polynomial spaces. This learning scheme exploits the low-dimensional structure of the problem, where the unknown function $\phi$ is a radial function defined on $\mathbb{R}^+$ regardless of the value of $N$. However, this approach leads to an inverse problem that requires solving for the identifiability of $\phi$, which is essential for analyzing the statistical optimality of least squares estimators.

In this paper, we approach the identifiability of $\phi$ in \eqref{system} from a linear statistical inverse learning problem perspective \cite{feng2021data}. That is, the observational  data comes in the form of positions of the $N$ particles on the manifold and the (possible noisy) first derivatives of these positions with respect to time, and they follow a joint probability distribution. To simplify notation, we write \eqref{system} as
\begin{equation}
\dot{\mathbf{X}}(t) = \rhsfo_{\phi}(\mathbf{X}(t))
\end{equation} where $\mathbf{X}=(\bx_1,\cdots,\bx_N)\in \mathcal{M}^N$ and $\rhsfo_{\intkernel}(\bX) \in T_{\mathbf{X}}\mathcal{M}^N$. We denote $\mathcal M^N=\mathcal M\times \cdots\times \mathcal M$ the product manifold consisting of $N$ copies of $\mathcal M$, with $T_{\bf X}\mathcal M^N$ the tangent space of $\mathcal M^N$ at ${\bf X}$. Let $A$ be a linear operator that maps a function $\varphi$ to $\rhsfo_{\varphi}$. Our study addresses two fundamental questions:
  \begin{itemize}
    \item[] \textbf{Question 1.}  What are  function spaces $\mathcal{H}$ and $\mathcal{F}$ such that  $A\in \mathcal{B}(\mathcal{H},\mathcal{F})$ (bounded linear operators mapping  $\mathcal{H}$ to $\mathcal{F}$)?
    
    \item[] \textbf{Question 2.} Suppose $A \in \mathcal{B}(\mathcal{H},\mathcal{F})$, when does the linear inverse problem 
    \begin{align}
    A\varphi = \rhsfo_{\varphi}
\end{align} have a stable solution?
    \end{itemize}

 We restrict our attention to $\mathcal{F}=L^2(\mathcal{M}^N;\mu;T_{\mathbf{X}}\mathcal{M}^N)$ where $\mu$ is the marginal distribution of positions  and $\mathcal{H}$ is a subspace of $L^2([0,R];\rho;\mathbb{R})$ with a suitable choice of measure $\rho$ constructed from $\mu$.  The use of $L^2$ spaces in both $\mathcal{H}$ and $\mathcal{F}$ is a natural choice when dealing with the $\ell^2$ loss function used in the learning. The space $\mathcal{F}$ is a space of square-integrable vector fields on $\mathcal{M}^N$ with respect to a probability measure $\mu$, which is the marginal distribution of observational position data. 
  The space $\mathcal{H}$ is a space of square-integrable functions supported on an interval $[0,R]$ so that the ODE system \eqref{system} is well-posed. For example, if $\mathcal{M}$ is compact, one can choose $R$ to be the injectivity radius of $\mathcal{M}$, see Section 2 in \cite{Maggioni2021LearningIK} and \cite{OpinionDynamicsongeneralcompactmanifold}. %
The answers to Question 1 and 2 have significant practical implications for designing statistical learning methods for predicting trajectories. In particular, the boundedness of $A$ means that the norm of $\mathcal{H}$ can be used as an effective error metric for estimators, allowing for the production of faithful approximations of the true velocity field. This is essential for accurate trajectory prediction. Additionally, the stability of the minimizer allows for the use of numerically stable procedures for obtaining faithful approximations of the minimizer from a finite amount of noisy data. Such results are crucial in establishing the statistical optimality of computation methods.  Our findings provide a positive answer to the conjectured geometric coercivity condition in \cite{Maggioni2021LearningIK} where the measure $\mu$ is the distribution of noise-free trajectory data with randomized initial distributions, but our framework can be applied  to any measure $\mu$ obtained in a more general observation regime, for example, taking the observation noise into account \cite{feng2021data}.

\paragraph{Related work and contribution}

Ordinary differential equations (ODEs) are a well-established tool for modeling dynamic processes with broad applications in various scientific fields. However, the explicit forms of these equations often remain elusive. The impressive advancement in data science and machine learning prompts researchers to devise methods originated from machine learning for data-driven discovery of dynamics. These examples include symbolic regression (e.g.\cite{schmidt2009distilling}), sparse regression/optimization (e.g.\cite{brunton2016discovering,schaeffer2018extracting,messenger2021weak}), kernel and Bayesian methods (\cite{raissi2017machine,heinonen2018learning,yang2021inference}), deep learning (e.g.\cite{qin2019data,du2022discovery}. The identifiability analysis is the first step in determining unknown parameters in ODE models. Previous work has mainly focused on parameter estimation. For instance, the identifiability analysis from single trajectory data for linear dynamical systems with linear parameters is performed in \cite{stanhope2014identifiability}, and later generalized to affine dynamical systems in \cite{duan2020identification}. For nonlinear ODEs arising from biomedical applications, one can refer to the survey in \cite{miao2011identifiability} for identifiability analysis methodologies for nonlinear ODE models and references therein.

Our study focuses on the application of data-driven modeling to complex particle systems on manifolds, and aims to conduct identifiability analysis motivated by the usage of nonparametric statistical learning methods. Specifically, we extend previous work on identifying interaction kernels in Euclidean spaces, as presented in \cite{lu2019nonparametric,lu2021learning,li2021identifiability}, to the more challenging manifold case. We explore how the geometry of the interacting domain affects the learning process. We adopt the newly developed  linear statistical inverse problem framework in \cite{feng2021data}, which allows for the analysis of data distributions from a broad range of observation regimes, in contrast to the previous work that only considered noise-free observations. Our work makes the following contributions:
\begin{itemize}
\item We find the solutions of \textbf{Question} 1 and \textbf{Question} 2 on  general manifolds, thereby providing a framework that includes the previous results as special cases.

\item Using our analysis, we refine the generic results on specific manifolds, such as the sphere and Hyperbolic space, providing a more nuanced understanding of the behavior of these systems.

\item Additionally, we investigate cases where the stability conditions may fail, deepening our understanding of the limitations and potential pitfalls of our approach.
\end{itemize}

Considering the manifold case has significant implications for various fields such as physics, biology, and social sciences, where systems can be modeled as interacting particles or agents on a manifold. Our approach provides a computational foundation for the data-driven discovery of these systems, leading to new insights and discoveries in these fields.

\section{Notation and Preliminaries}

\paragraph{Notation:} Let $\nu$ be a Borel positive measure on a  domain $\mathbb{D}_1$.  We use $L^2(\mathbb{D}_1;\nu;\mathbb{D}_2)$ to denote the set of $L^2(\nu)$-integrable vector-valued functions that map $\mathbb{D}_1$ to $\mathbb{D}_2$. 

Let $\mathcal{S}_1$ be  a  measurable subset of $\mathbb{R}^{D}$, then the restriction of the measure $\rho$ on $\mathcal{S}_1$, denoted by $\rho\mres \mathcal{S}_1$, is defined as $\rho\mres \mathcal{S}_1(\mathcal{S}_2)=\rho(\mathcal{S}_1 \cap \mathcal{S}_2)$ for any measurable subset $\mathcal{S}_2$ of $\mathbb{R}^{D}$. 
For two Borel positive measures $\rho_1, \rho_2$ defined on $\mathbb{R}^{D}$, $\rho_1$ is said to be absolutely continuous with respect to $\rho_2$, denoted by $\rho_1 \ll \rho_2$, if $\rho_1(\mathcal{S}) = 0$ for every set $\rho_2(\mathcal{S}) = 0$, $\mathcal{S} \subset \mathbb{R}^{D}$. $\rho_1$ and $\rho_2$ are called equivalent iff $\rho_1 \ll \rho_2$ and $\rho_2 \ll \rho_1$.

Let $\mathcal{H}_1, \mathcal{H}_2$ be  Hilbert spaces. We use $\langle\cdot,\cdot\rangle_{\mathcal{H}_1}$ to denote the inner product over $\mathcal{H}_1$, and still use $\langle\cdot,\cdot\rangle$ to denote the inner product on the Euclidean space.  We denote by $\mathcal{B}(\mathcal{H}_1,\mathcal{H}_2)$  the set of bounded linear operators mapping $\mathcal{H}_1$ to $\mathcal{H}_2$.
Let $A \in \mathcal{B}(\mathcal{H}_1,\mathcal{H}_2)$, we use $\|A\|$ to denote its operator norm.

\subsection{Preliminaries in Geometry}

We will use several geometric tools. For the convenience of the reader, we state them here. More details of these can be found in various books such as \cite{DoCarmo} and \cite{GallotRiemannian}.

We let $(\mathcal M,{g})$ denote a smooth, connected, and complete $n$-dimensional Riemannian manifold with Riemannian metric $g$. For any point $\bx \in \mathcal M,$ we let $T_{\bx}\mathcal M$ be the tangent space at $\bx$ and let 
\begin{align*}
\exp_{\bx}:T_{\bx}\mathcal M\rightarrow\mathcal M
\end{align*}denote the {\it exponential map}. If $\gamma_{\bx,{\bf v}}$ is the geodesic starting at $\bx$ with initial velocity ${\bf v}$, then $\exp_{\bx}({\bf v})=\gamma_{\bx,{\bf v}}(1)$. Let $I(\bx,\mathbf{v})>0$ be the first time when the geodesic $s\mapsto \exp_{\bx}(s\mathbf{v})$ stops to be minimizing. Define 
$$\mathcal D(\bx):=\{t\mathbf{v} \in T_{\bx} \mathcal M: \mathbf{v}\in \mathcal S^{n-1},\,\,0\leq t< I(\bx,\mathbf{v})\},$$
where $\mathcal S^{n-1}$ is the unit sphere in the tangent space with respect to the metric $g(\bx)$, then $\exp_{\bx}$ is a diffeomorphism from $\mathcal D(\bx)$ onto its image, and
\begin{align*}
    d(\bx,\mathbf{p})=|\exp_{\bx}^{-1}(\mathbf{p})|
\end{align*}
for any ${\bf p}\in \exp_{\bx}(\mathcal D(\bx))$. Recall that $d(\bx,{\bf p})$ is the Riemannian distance between $\bx$ and ${\bf p}$, which is obtained by minimizing the length of
all piecewise $C^1$ curves connecting $\bx$ and ${\bf p}$.
Let $\mathcal{CL}(\bx):=\exp_x(\partial \mathcal D(\bx))\subset \mathcal M$ be the {\it cut locus} of $\bx$, it is known that for connected complete Riemannian manifolds
$$\mathcal M=\exp_{\bx}(\mathcal D(\bx))\cup \mathcal{CL}(\bx)$$
for any $\bx\in\mathcal M$, and $\mathcal{CL}(\bx)$ has zero measure, see, e.g. [\cite{GallotRiemannian}, page 158, Corolarry 3.77]. 
We denote $\textup{Inj}(\bx):=d(\bx,\mathcal{CL}(\bx))>0$ to be the {\it injectivity radius} at $\bx\in \mathcal M$ (note that if $\mathcal{CL}(\bx)=\emptyset$, $\textup{Inj}(\bx)=\infty$), and let $\textup{Inj}(\mathcal{M})=\inf_{\bx\in \mathcal M}\textup{Inj}(\bx)$ be the injectivity radius of $\mathcal M$. {If $\mathcal M$ is compact, then $\textup{Inj}(\mathcal M)$ is strictly positive.}
 
It is easy to check that $\mathcal D(\bx)$ is an open star-shaped subset of $T_{\bx}\mathcal M$, so $(\mathcal D(\bx), \exp_{\bx})$
is a smooth chart and one has that for any Borel measurable function $f$ on $\mathcal M$ 
\begin{align}\label{integral on M}
    \int_{\mathcal M}f\,dV_\mathcal{M}=\int_{\mathcal D(\bx)} (f\circ \exp_{\bx})\sqrt{\det (G_{\bx})}\,d\lambda^n ,
\end{align}
where $dV_\mathcal{M}$ is the Riemannian measure on the manifold $(\mathcal M,g)$, $G_{\bx}:=\exp_{\bx}^* g$ is the pullback of $g$ by $\exp_{\bx}$, and $\lambda^n$ denotes the $n$ dimensional Lebesgue measure on $\mathbb R^n.$ Notice that given coordinates $\{u^1,\cdots, u^n\}$ of $\mathcal D(\bx)$, for any $u\in\mathcal D(\bx)$
$$(\exp^*_{\bx}g)_u(\frac{\partial}{\partial u^i},\frac{\partial}{\partial u^j})=g_{\exp_{\bx}(u)}(d\exp_{\bx}\big|_u(\frac{\partial}{\partial u^i}),d\exp_{\bx}\big|_u(\frac{\partial}{\partial u^j})),$$
where $d\exp_{\bx}\big|_u:T_u\mathcal D(\bx)\to T_{\exp_{\bx}(u)}\mathcal M$ is the differential of the exponential map.

As an explicit example to illustrate these concepts, let us consider the sphere $\mathcal M=\mathbb S^n\subset \mathbb R^{n+1},$ where one finds that   $$T_{\bx}\mathbb S^n=\{\mathbf{v}\in \mathbb R^{n+1}: \langle \bx,\mathbf{v}\rangle =0\}=:\bx^\perp,$$ where $\langle\cdot ,\cdot \rangle $ denotes the standard inner product on $\mathbb R^{n+1}.$ We choose the round metric on $\mathbb S^n$ which is the pullback of the Euclidean metric of $\mathbb R^{n+1}.$ In this metric, one can compute the distance between two points as follows for any $\bx, \mathbf y\in \mathbb S^n\subset \mathbb R^{n+1},$ we have 
\begin{align}\label{distance-on-sphere}
    d_{\mathbb{S}^n}(\mathbf{x},\mathbf{y})=\arccos(\langle \bx,\mathbf y\rangle).
\end{align}Note that for this metric the exponential chart can be computed explicitly, being 
\begin{align}\label{exponentialmaponsphere}
\exp_{\bx}(\mathbf{v})=\cos (\|\mathbf{v}\|)\mathbf{x}+\sin(\|\mathbf{v}\|)\frac{\mathbf{v}}{\|\mathbf{v}\|}\quad \textup{for }\mathbf{v}\in T_{\bx}\mathbb S^n=\bx^\perp.
\end{align}

The cut domain is given by $\mathcal D(\bx)=B_\pi(\mathbf{0})$ for any $\bx \in \mathbb S^n,$ i.e., the ball of radius $\pi$ centered at origin. Writing the metric in polar coordinates $g_{\mathbb S^n}= dr^2+\sin^2(r)g_{\mathbb S^{n-1}}$, we thus we get for any measurable $f$ on $\mathbb S^n$ that 
\begin{align}\label{integralonsphere}
    \int_{\mathbb S^n}f\,dV_{\mathbb{S}^n}=\int_0^\pi\int_{\mathbb S^{n-1}}f(\exp_{\mathbf{x}}(t\mathbf{v})){\sin^{n-1}(t)}\, d\mathcal \sigma(\mathbf{v})\,dt,
\end{align}
where $\mathbb S^{n-1}$ denotes the $n-1$ dimensional unit sphere and $d\sigma $ denotes the induced surface measure.

The other simple case is the hyperbolic space, $\mathbb H^n=\{\mathbf{x} =(x_1,\cdots, x_n) 
\in \mathbb R^n: x_n>0\}$ with the usual metric $g_{\mathbb H^n}=\frac{1}{x_n^2}g_{\textrm{Euclid}}.$ By the Hadamard theorem, we have that $\mathcal{CL}(\mathbf{p})=\emptyset$ for any $\mathbf{p} \in \mathbb H^n,$ and that $\mathcal D(\mathbf{p})=\mathbb R^n.$ The equation \eqref{integralonsphere}  for $\mathbb{H}^n$ is the same except one has to change the area of integration and  replaces $\sin$ by $\sinh:$
\begin{align}\label{integralonhyperbolic}
    \int_{\mathbb H^n}f\,dV_{\mathbb{H}^n}=\int_0^\infty\int_{\mathbb S^{n-1}}f(\exp_{\mathbf{x}}(t\mathbf{v})){\sinh^{n-1}(t)}\, d\mathcal \sigma(\mathbf{v})\,dt.
\end{align}

\subsection{Interactions on manifold} \label{sec:pre:manifold}
Following the work in  \cite{Maggioni2021LearningIK,OpinionDynamicsongeneralcompactmanifold},
we choose  $d(\bx_i, \bx_j)$ to be the  geodesic distance of any two points $\bx_i,\bx_j\in \mathcal M$. 
Notice that we can connect $\bx_i$ and $\bx_j$ with a minimizing geodesic which is unique if they are not cut points of each other. In view of that we define the influence vector as  
\begin{align}\label{ApproachB}
    w(\bx_i,\bx_j)=\begin{cases} d(\bx_i,\bx_j) \dot \gamma_{\bx_i,\bx_j}(0)\quad \textup{if } \bx_i \neq \bx_j, \textup{ and }\bx_i \notin \mathcal{CL}({\bx_j})\\
    0\quad \textup{otherwise,}
    \end{cases}
\end{align}
where $\gamma_{\bx_i,\bx_j}:[0,d(\bx_i,\bx_j)]\rightarrow \mathcal{M}$ denotes the unique normalized minimizing geodesic connecting $\bx_i$ and $\bx_j.$ So $\dot{\gamma}_{\bx_i,\bx_j}(0)$ is the unit tangent vector to the geodesic $\gamma_{\bx_i,\bx_j}$ at $\bx_i$. In particular, $w(\bx,\by)=\exp_{\bx}^{-1}(\by)$ for $\by\notin \mathcal{CL}(\bx)$.

\begin{remark} 

\begin{itemize} We add several comments on the first-order models. 

\item It is sometimes conventional to consider the unit norm influence vector by moving the term $d(\bx_i,\bx_j)$ in  \eqref{ApproachB} to the interaction part, so the interaction kernel becomes $\phi(\cdot)\cdot$. The analysis developed in this paper can be equivalently applied to this case as well by slight modification on the function spaces.

\item When $\mathcal{M}=\mathbb{R}^n$ or $\mathbb{S}^n$, we can calculate the influence vector as follows
\begin{align*}
    w(\bx_i,\bx_j)=\begin{cases}{\Pi_{T_{\bx_i}\mathcal M}(\bx_j-\bx_i)}\quad \textup{if } \Pi_{T_{\bx_i}\mathcal M}(\bx_j-\bx_i)\neq 0\\
    0\quad \textup{otherwise}.
    \end{cases}
\end{align*} where $\Pi_{T_{\mathbf{x}_i}\mathcal M} $ denotes the orthogonal projection onto the tangent space $T_{\bx_i}\mathcal M$. In particular, when $\mathcal{M} = \mathbb{R}^n$, 
one has that the vectors of influence are simply
\begin{align*}
w(\bx_i,\bx_j)=\bx_j-\bx_i. 
\end{align*} This recovers the classical first order models from \cite{motsch2014heterophilious}.
\item In \cite{OpinionDynamicsongeneralcompactmanifold}, the well-posedness of \eqref{system} is examined for compact manifolds. The well-posedness for the mean-field equation of \eqref{system} is examined in \cite{fetecau2018self} on two dimensional sphere and two dimensional hyperbolic spaces. It is established that if the support of the interaction kernel is smaller than the injective radius, then \eqref{system} is well-posed. However, it is worth noting that our analysis does not necessitate the manifold being compact.
\end{itemize}
\end{remark}

We let $\mathcal{M}^N = 
\mathcal{M}\times\mathcal{M}\times\cdots \times\mathcal{M}$ be the canonical product of Riemannian manifolds with product Riemannian metric given by $g_{\mathcal{M}}^N$
. Denote by
\begin{align}
    \bX: = \begin{bmatrix}
    \vdots\\\bx_i \\ \vdots
    \end{bmatrix} \in \mathcal{M}^N, 
\end{align} the tangent space at $\bX$  is defined by $T_{\mathbf{X}}\mathcal{M}^N =T_{\bx_1}\mathcal M\times\cdots\times T_{\bx_N}\mathcal M$ such that 
$$\bigg\langle \begin{bmatrix}
    \vdots\\\mathbf{u}_i \\ \vdots
    \end{bmatrix}, \begin{bmatrix}\vdots\\\mathbf{v}_i \\ \vdots \end{bmatrix}\bigg\rangle_{T_{\mathbf{X}}\mathcal{M}^N}:= \frac{1}{N}\sum_{i=1}^{N}g\left( \mathbf{u}_i,\mathbf{v}_i\right)$$
which is the weighted product metric induced by the metric $g$.

\subsection{Preliminaries on probability}

We say that $N$ random vectors $X_1,\cdots,X_N$ are \textbf{conditional i.i.d} if there exists a random vector $\theta \sim \pi(\theta)$ such that $X_1|\theta,\cdots, X_N|\theta$ are independent and identically distributed.  Then
\[p(X_1,\cdots,X_N)=\int p(X_1|\theta)p(X_2|\theta)\cdots p(X_N|\theta)\pi(\theta)d\theta,\]
It is also called a \textbf{mixture} of i.i.d. For example,  $X_1,\cdots,X_N$ is conditional i.i.d when they are from an infinite exchangeable sequence, which is a consequence of  the celebrated de Finetti's representation theorem. It also includes $N$ i.i.d random vectors as a special example. 

\textcolor{black}{\textbf{Examples of $\mu$:} In this paper, we assume that the observational data, consisting of pairs of positions and velocities, follows a probability distribution, and $\mu$ represents the marginalization of this joint distribution with respect to the position data. In \cite{Maggioni2021LearningIK, lu2019nonparametric, lu2021learning, feng2021data}, the authors explore the scenario involving multiple i.i.d.\ trajectory data, with initial positions sampled i.i.d.\ from a probability measure $\gamma$. In this framework, if observations occur only at $t=0$ (e.g., see Theorem 9 in \cite{lu2021learning}), then $\mu$ can be equated to $\gamma$. On the other hand, in scenarios where one observes the steady state of particle dynamics, such as configurations constituting vertices of regular complexes often seen in empirical phenomena (see an example in Section \ref{singularexample}), $\mu$ could become discrete.
 }

\section{On the Boundedness of \textit{A}}
We begin by defining a measure $\rho$ on the Borel $\sigma$-algebra of $\mathbb{R}^{+}$, which addresses Question 1. Throughout the paper, if not specified otherwise, $\mu$  defined on $\mathcal{M}^N$ is assumed to be absolutely continues with respect to the volume form. 
\begin{definition}\label{themeasure}
Let $\mathbf{X}=(\bx_1,\cdots,\bx_N) \in \mathcal{M}^N$ and let $B\subset \mathbb{R}^{+}$ be a Borel set. Then, we define the positive measure $\rho$ as follows:
\begin{align*}
\rho(B)
=\frac{1}{N(N-1)}\sum_{i \neq j}\int_{P^{-1}_{ij}(B)}P^2_{ij}(\mathbf{X}) d\mu(\mathbf{X}),
\end{align*}
where the function $P_{ij}(\mathbf{X}):= d(\bx_i,\bx_j)$ maps $\mathcal{M}^N$ to $\mathbb{R}^+$. In other words, $\rho(B)$ is the average of measures that are push-forwards of $\mu$ by the distance map $P_{ij}^2$. \end{definition}

 \textcolor{black}{The measure $\rho$ can be conceptualized as an ``occupancy" measure. More specifically, for any interval $I$, $\rho(I)$ quantifies the likelihood of observing a pair of agents whose distance apart falls within $I$. This measure is integral to understanding the dynamics within highly explored regions—areas where agent interactions are frequent and thus provide a richer dataset for estimating the interaction kernel with greater accuracy.} In a nutshell, the measure $\rho$ serves to capture the structural information of the governing equations \eqref{system} by quantifying the extent to which regions of $\mathbb{R}^+$ are explored through samples from $\rhsfo_{\phi}(\bX)$ at the population level.

\textcolor{black}{Given a density function $\mu$, directly deriving an analytical expression for $\rho$ can be challenging. A simple case is when $\mathcal{M}=\mathbb{R}$ and $\mu$ is the uniform distribution on $[0,1]^N$, then  $\rho$ is compactly supported on $[0,1]$ with density function $2(1-t)$ for $0\leq t\leq 1$. One can also refer to a derivation of $\rho$ in the Euclidean case when $\mu$ is Gaussian (see Theorem 9 in \cite{lu2021learning}). In general, we can employ Monte Carlo simulations and density estimation techniques to approximate $\rho$. We have elaborated on this methodological approach and provided an illustrative example in Figure \ref{fig:my_label_2}, which demonstrates the practical application of these techniques in estimating $\rho$ based on specified $\mu$ values.}

By virtue of being the average of measures that are push-forwards of $\mu$ by the distance map $P_{ij}^2$, $\rho$ inherits some properties of $\mu$. For instance, if $\mu$ is a Borel regular measure on $\mathcal{M}^N$, then $\rho$ is also Borel regular on $\mathbb{R}^+$. Given our assumption that $\mu$ is absolutely continuous with respect to the volume measure $V_{\mathcal{M}^N}$ of $\mathcal{M}^N$, it follows that $\rho$ enjoys the same property, as succinctly outlined below:

\begin{lemma}\label{Lemma21}
 The measure $\rho$ in Definition \ref{themeasure} is absolute continuous with respect to the Lebesgue measure of $\mathbb R_+$. 
\end{lemma}
\begin{proof} It suffices to show that for any $i,j,$ the measure $\rho_{ij}$ given by \begin{align*}
    \rho_{ij}(A)=\int_{P_{ij}^{-1}(A)} P^2_{ij}(\boldsymbol{X})\, d\mu(\boldsymbol{X})
\end{align*}
is absolutely continuous with respect to $|\cdot|, $ the one dimensional Lebesgue measure. 
 Let $A\subset \mathbb R^+$ satisfying $|A|=0$. We denote $A^2=\{x^2: x\in A\}$, then $|A^2|=0$ too. Notice that $P_{ij}^2: \mathcal M^N\to \mathbb R^+$ is smooth, and the set of critical points (i.e. with zero gradient) of $P^2_{ij}$ is 
 $$\{{\bf X}=(\bx_1,\cdots,\bx_N)\in \mathcal M^N: \bx_i=\bx_j\},$$
which is of zero measure, therefore $V_{\mathcal{M}^N}((P^2_{ij})^{-1}(A^2))=0$ whenever $|A^2|=0$. It follows that $V_{\mathcal{M}^N}(P^{-1}_{ij}(A))=V_{\mathcal{M}^N}((P^2_{ij})^{-1}(A^2))=0$. Since $\mu$ is absolutely continuous with respect to $V_{\mathcal{M}^N}$, we obtain that $\mu(P^{-1}_{ij}(A))=0$, thus $\rho_{ij}(A)=0$. This shows that $\rho$ is absolute continuous with respect to the Lebesgue measure $|\cdot|$ on $\mathbb R^+$.
\end{proof}

In summary, the measure $\rho$ is derived from $\mu$ and shares numerous properties with it, rendering it a valuable instrument in analyzing the operator $A$ and characterizing the function space $\mathcal{H}$. With this groundwork laid, we are prepared to present the key outcome of this section.

\begin{proposition}\label{lem1}
Let $\mathcal{H}$ be  a subspace of $L^2([0,R];\rho\mres[0,R];\mathbb{R})$ or itself. Then $A \in \mathcal{B}(\mathcal{H},\mathcal{F})$ and $\|A\|\leq \frac{N-1}{N}.$
\end{proposition}
\begin{proof} It suffices to prove this result when $\mathcal{H}:=L^2([0,R];\rho\mres[0,R];\mathbb{R})$.
Observe that for any $\varphi \in \mathcal{H},$ one has that 
\begin{align*}
    \|\rhsfo_\varphi \|^2_{L^2(\mu)}=\frac{1}{N}\sum_{i=1}^N\int_{\mathcal M^N}\|\frac{1}{N}\sum_{j\neq i}\varphi(d(\bx_{i},\bx_j))\omega(\bx_i,\bx_j)\|^2\, d\mu(\boldsymbol X).
\end{align*}

We now estimate the integrand using Cauchy-Schwartz and Young's inequality: for any $i=1,\dots , N$
\begin{align*}
    \|\frac{1}{N}\sum_{j\neq i}\varphi(d(\bx_i,\bx_j))\omega(\bx_i,\bx_j)\|^2
    &=g\Bigl(\frac{1}{N}\sum_{j\neq i}\varphi(d(\bx_{i},\bx_j))\omega(\bx_i,\bx_j),\frac{1}{N}\sum_{k\neq i}\varphi(d(\bx_{i},\bx_k))\omega(\bx_i,\bx_k)\Bigr)\\
    &=\frac{1}{N^2}\sum_{j\neq i}\sum_{k\neq i}\varphi(d(\bx_{i},\bx_j))\varphi(d(\bx_{i},\bx_k))g\left(\omega(\bx_i,\bx_j),\omega(\bx_i,\bx_k)\right) \\
    &\leq \frac{1}{N^2}\sum_{j\neq i}\sum_{k\neq i}\big|\varphi(d(\bx_{i},\bx_j))\varphi(d(\bx_{i},\bx_k))\big|d(\bx_{i},\bx_j)d(\bx_{i},\bx_k)\\
    &\leq \frac{1}{2N^2}\sum_{j\neq i}\sum_{k\neq i}\varphi^2(d(\bx_{i},\bx_j))d^2(\bx_i,\bx_j)+\varphi^2(d(\bx_{i},\bx_k))d^2(\bx_i,\bx_k)\\
    &=\frac{N-1}{N^2}\sum_{j\neq i}\varphi^2(d(\bx_{i},\bx_j))d^2(\bx_i,\bx_j).
\end{align*}
From this estimate we get that 
\begin{align*}
    \|\rhsfo_\varphi \|^2_{L^2(\mu)}\leq \frac{N-1}{N^3}\sum_{i\neq j}\int_{\mathcal M^N}\varphi^2(d(\bx_{i},\bx_j))d^2(\bx_i,\bx_j)\, d\mu(\boldsymbol X)
\end{align*}
Moreover, one has that 
\begin{align*}
    \|\varphi\|^2_{L^2(\rho)}=\frac{1}{N(N-1)}\sum_{i\neq j}\int_{\mathcal M^N}\varphi^2(d(\bx_{i},\bx_j))d^2(\bx_i,\bx_j)\, d\mu(\boldsymbol X),
\end{align*}
from which we get that 
\begin{align*}
     \|\rhsfo_\varphi \|^2_{L^2(\mu)} = \|A\varphi \|^2_{L^2(\mu)}\leq \frac{(N-1)^2}{N^2} \|\varphi\|^2_{L^2(\rho)},
\end{align*}
as desired.
\end{proof}

The above theorem establishes that the $L^2$ norm induced by $\rho$ is a potent metric for data-driven estimators. Estimators that yield small $L^2(\rho)$ errors provide accurate approximations of the velocity field, thereby guaranteeing reliable trajectory prediction outcomes. Notably, the $L^2(\rho)$ norm has been utilized as an evaluation metric for the generalization error of least square estimators in \cite{Maggioni2021LearningIK}.

\section{Well-posedness}

In this section, we present stability results for a specific family of measures, namely the joint distribution of $N$ independent and identically distributed random variables on $\mathcal M$. In the forward problem, it is often assumed that the initial conditions of particles are i.i.d. because these particles are indistinguishable. Under certain conditions, such as the absence of significant interactions between the particles and $N$ is sufficiently large, the system may remain approximately i.i.d. over time. In our analysis, we leverage the i.i.d assumption to facilitate computation and prove that the associated integral operator is positive and even strictly positive. We later demonstrate that the stability results hold more generally for the joint distribution of $N$ conditionally i.i.d. random variables by leveraging De Finetti's representation theorem. Finally, we show that the stability results hold for their equivalent measures in Corollary \ref{cor}.

\begin{proposition}\label{mainresult}
 Suppose that $\mu $ is the joint distribution of  $N$ i.i.d random variables. Then we have that $\forall \varphi\in \mathcal{H} =L^2([0,R];\rho\mres[0,R];\mathbb{R})$
\begin{align*}
  \|A\varphi\|_{L^2(\mu)} \geq \frac{\sqrt{N-1}}{N}\|\varphi\|_{L^2(\rho)}
\end{align*}
\end{proposition}
\begin{proof}
We calculate that 
\begin{align*}
  &\left\|A\phi \right\|_{L^2(\mu)}^2 =  \mathbb E_{\bX \sim \mu}\left[\left\|\rhsfo_{\phi}(\bX) \right\|_{T_{\bX} \mathcal{M}^N}^2\right] 
     \\&= \mathbb E_{\bX\sim \mu} \sum\limits_{i=1}^N \frac1N g\left( \frac1N \sum\limits_{j\neq i}\phi(P_{ij}(\boldsymbol{X})) \omega(\bx_i,\bx_j), \frac1N \sum\limits_{k\neq i} \phi(P_{ik}(\boldsymbol{X})) \omega(\bx_i,\bx_k) \right)\\
     &=\frac{1}{N^3}\sum_{j\neq i}\mathbb E_{\bX\sim\mu}\Bigl[ \phi^2(P_{ij}(\boldsymbol{X}))P_{ij}^2(\boldsymbol{X})\Bigr] \\&+\frac{1}{N^3}\sum_{j\neq i, k\neq j, k \neq i}\mathbb E_{\bX\sim\mu}\Bigl[ \phi(P_{ij}(\boldsymbol{X}))\phi(P_{ik}(\boldsymbol{X}))g\Bigl(  \omega(\bx_i,\bx_j), \omega(\bx_i,\bx_k)\Bigr)\Bigr],
\end{align*}
where we used the fact that in view of \eqref{ApproachB}, one has that $\|\omega(\bx_i,\bx_j)\|^2=P_{ij}^2(\boldsymbol{X}).$
Note that 
\begin{align*}
 &  \frac{1}{N^3}\sum_{j\neq i}\mathbb E_{\bX\sim\mu}\Bigl[ \phi^2(P_{ij}(\boldsymbol{X}))P^2_{ij}(\boldsymbol{X})\Bigr]\\
 &=\frac{N(N-1)}{N^3}\int_{\mathcal{M}^N} \frac{1}{N(N-1)}\sum_{i\neq j}\phi ^2(P_{ij}(\boldsymbol{X}))P^2_{ij}(\boldsymbol{X})\, d\mu(\boldsymbol{X})\\
 &=\frac{N-1}{N^2}\|\phi \|^2_{L^2(\rho)}.
\end{align*}
This gives that 
\begin{align*}
     &  \left\|A\phi \right\|_{L^2(\mu)}^2
     \\ &=\frac{N-1}{N^2}\|\phi \|^2_{L^2(\rho)}+\frac{1}{N^3}\sum_{i\neq j\neq k} \mathbb E_{\bX\sim\mu}\Bigl[ \phi(P_{ij}(\boldsymbol{X}))\phi(P_{ik}(\boldsymbol{X}))g\Bigl(  \omega(\bx_i,\bx_j), \omega(\bx_i,\bx_k)\Bigr)\Bigr].
\end{align*}
It therefore suffices to show that for $i\neq j\neq k$ (any three non-equal indices)
\begin{align*}
     C_{ijk}=\mathbb E_{\bX\sim \mu}\Bigl[ \phi(P_{ij}(\boldsymbol{X}))\phi(P_{ik}(\boldsymbol{X}))g\Bigl(  \omega(\bx_i,\bx_j), \omega(\bx_i,\bx_k)\Bigr)\Bigr]\geq 0.
\end{align*}

Due to the special structure of $\mu$, i.e. $\mu=\mu_0\times\cdots\times\mu_0$ with $\mu_0$ the same probability measure for each $\bx_i$, $i=1,\dots, N$, the values of $C_{ijk}$'s are equal and therefore we just need to show it is positive for one such term. Without loss of generality, let $i=1$, $j=2$ and $k=3$, and we further simplify the notations by denoting $\bx:=\bx_1$, $\by:=\bx_2$ and $\bz:=\bx_3$.
We then get that 
\begin{align}\label{C123}
   C_{123}&=\mathbb E_{\bx,\by,\bz}\Bigl[ g\Bigl(  \phi(d(\bx,\by)) w(\bx,\by),\phi(d(\bx,\bz)) w(\bx,\bz) \Bigr)\Bigr ]\nonumber\\
   & =\mathbb E_{\bx}\mathbb E_{\by, \bz}\Bigl[\phi(d(\bx,\by))\phi(d(\bx,\bz)) g(w(\bx,\by),w(\bx,\bz))\Bigr| \bx\Bigr],
\end{align} where the second equality is true due to the independence of $\bx,\by$ and $\bz$.

It therefore suffices to show that for any fixed $\bx,$ we have that 
\begin{align*}
   \mathbb E_{\by, \bz}\Bigl[\phi(d(\bx,\by))\phi(d(\bx,\bz)) g(w(\bx,\by),w(\bx,\bz))\Bigr| \bx\Bigr]\geq 0.
\end{align*}
Now for each fixed $\bx$, $\phi(d(\bx,\by))\phi(d(\bx,\bz))$ and $g(w(\bx,\by),w(\bx,\bz))$ are positive definite kernels with respect to $(\by,\bz)\in \mathcal M\times\mathcal M$ by Theorem \ref{t52} (2)(3) (notice that $g$ is positive definite on $T_{\bx}\mathcal M\times T_{\bx}\mathcal M$), therefore their product is a positive definite kernel by Theorem \ref{t52} (1). Then by (4) of Theorem \ref{t52}, the conclusion follows.
\end{proof}

\begin{remark}In fact, the same analysis can be extended to cover the case of $N$ conditionally i.i.d. random variables, by invoking de Finetti's representation theorem. Specifically, in \eqref{C123},given three conditionally i.i.d random variables $\bx,\by,\bz$, one can do 

\begin{align}
   C_{123}=\mathbb \mathbb{E}_{\theta}E_{\bx|\theta,\by|\theta,\bz|\theta}\Bigl[ g\Bigl(  \phi(d(\bx|\theta,\by|\theta)) w(\bx|\theta,\by|\theta),\phi(d(\bx|\theta,\bz|\theta)) w(\bx|\theta,\bz|\theta) \Bigr)\Bigr |\theta \Bigr]
\end{align} where $\theta$ is the latent random variable conditioning on which that  $\bx|\theta,\,\by|\theta,\,\bz|\theta$ are independent. Then with  a slight modification on the proof, we can show that $C_{123} \geq 0$. 
\end{remark}

         \subsection{Discussion on the sharpness of the bound {$\frac{\sqrt{N-1}}{N}$} on specific manifolds}
\textcolor{black}{The bound $\frac{\sqrt{N-1}}{N}$ exhibits a somewhat counterintuitive characteristic: it decreases as $N$ increases. At first glance, one might expect that more data, represented by a larger $N$, would naturally enhance the learning process due to the increased number of equations in our model. Notably, our kernel remains a 1D function, consistent across all these equations.}

\textcolor{black}{However, this initial intuition overlooks the complexities inherent in inverse problems, particularly in the context of dynamical systems where data dependencies can be intricate. This observation has prompted a deeper analysis of the underlying dynamics and their impact on our learning process.}

In this section, we explore whether this bound  is sharp or can be enhanced by either confining $\mathcal{H}$ to a smaller subspace, such as a compact subspace of $L^2([0,R];\rho\mres[0,R];\mathbb{R})$, or by considering specific manifolds that enable us to refine the outcomes.
We first discuss a criterion to decide whether, given a compact subset $\mathcal H\subset L^2([0,R];\rho\mres[0,R];\mathbb{R})$, there is going to be a larger bound in general or not, when restricting the map $A$ to it.

\begin{theorem}
\label{remark1}
Suppose that $\mu$ is the joint distribution of $N$ i.i.d random variables on $\mathcal M.$ Let $\mathcal H\subset L^2([0,R];\rho\mres[0,R];\mathbb{R})$ be a compact subspace (i.e., the unit ball of $\mathcal{H}$ is compact with respect to the $L^2(\rho)$ norm) consisting of continuous functions. If for any nonzero $\varphi \in \mathcal H,$ there exists some $\bx\in \textup{supp}(\mu)\subset \mathcal{M}^N$ such that 
\begin{align}\label{Positvityatx1}
     \mathbb E_{\by,\bz}\Bigl[ \varphi(d(\bx,\by))\varphi(d(\bx,\bz)) g\Bigl(\omega(\bx,\by), \omega(\bx,\bz) \Bigr)\Big|\bx\Bigr ]
    > 0,
\end{align}
Then one has that 
\begin{align*}
    \inf_{\varphi\in\mathcal H,\varphi 
    \neq 0} \frac{\|A\varphi\|^2_{L^2(\mu)}}{\|\varphi\|^2_{L^2(\rho)}}\geq \frac{N-1}{N^2}+c_{\mathcal H, \mathcal M}\frac{(N-1)(N-2)}{N^2}, 
\end{align*}
where  \begin{align}\label{psconstant}
       c_{\mathcal H,\mathcal M}:=\inf_{\varphi\in \mathcal H, \|\varphi\|_{L^2(\rho)}=1}  \mathbb E\Bigl[ \varphi(d(\bx,\by))\varphi(d(\bx,\bz)) g\Bigl(w(\bx,\by), w(\bx,\bz) \Bigr)\Bigr ]>0.
    \end{align}
\end{theorem}

\begin{proof}  We divide the proof into several steps.

\textbf{Step 1.} We claim that the function
 \begin{align*}
     \bx \mapsto \mathbb E_{\by,\bz}\Bigl[ \varphi(d(\bx,\by))\varphi(d(\bx,\bz)) g\Bigl(\omega(\bx,\by), \omega(\bx,\bz) \Bigr)\Big|\bx\Bigr ]
 \end{align*}
is continuous.
 Take a sequence $\{\bx_n\} \subset \mathcal{M}$ such that $\text{lim}_{n \rightarrow \infty}\bx_n=\bx$. Due to the continuity of $\varphi$, the (Lipschitz) continuity of $d(\cdot,\by),$ and $d(\cdot,\bz),$ respectively, we have 
 $\varphi(d(\bx_n,\by))\varphi(d(\bx_n,\bz))\rightarrow \varphi(d(\bx,\by))\varphi(d(\bx,\bz))$ as $n\rightarrow \infty$ for each $\by, \bz\in \mathcal{M}$. 
 
 Recall that $w(\bx,\by)=\exp^{-1}_{\bx}(\by)$, $w(\bx,\bz)=\exp^{-1}_{\bx}(\bz)$ for $\by,\bz\notin\mathcal{CL}(\bx)$. Moreover, if $\by,\bz\notin\mathcal{CL}(\bx)$, then $\by,\bz\notin\mathcal{CL}(\bx_n)$ for $n$ sufficiently large. Therefore, $\exp_{\bx_n}^{-1}(\by)\to \exp_{\bx}^{-1}(\by)$, $\exp_{\bx_n}^{-1}(\bz)\to \exp_{\bx}^{-1}(\bz)$ as $n\to\infty$, for $\by,\bz\notin\mathcal{CL}(\bx)$. Note that $\mathcal{CL}(\bx)$ has zero measure with respect to $\mu_0$, it follows that $g(w(\bx_n,\by),w(\bx_n,\bz))\to g(w(\bx,\by),w(\bx,\bz))$ a.e. on $\mathcal M\times\mathcal M$ with respect to the measure $\mu_0\times\mu_0$.
 \textcolor{black}{Note that we have 
 \begin{align*}
     \bigg|\varphi(d(\bx_n,\by))\varphi(d(\bx_n,\bz)) g\Bigl(\omega(\bx_n,\by), \omega(\bx_n,\bz) \Bigr)\bigg| \leq \|\varphi\|_\infty^2R^2
 \end{align*}
$\mu-$ almost everywhere thanks to the compact support of $\varphi$.} 
The claim then follows from dominated convergence theorem as the integrand is continuous in $\bx$ for almost every $(\by,\bz)\in \mathcal M\times \mathcal M.$


 We thus have that 
 \begin{align*}
     &\lim_{n\rightarrow \infty}\mathbb E_{\by,\bz}\Bigl[ \varphi(d(\bx_n,\by))\varphi(d(\bx_n,\bz)) g\Bigl(\omega(\bx_n,\by), \omega(\bx_n,\bz) \Bigr)\Big|\bx_n\Bigr ]\\
     =& \lim_{n \rightarrow \infty} \int_{\mathcal M \times \mathcal M}\varphi(d(\bx_n,\by))\varphi(d(\bx_n,\bz)) g\Bigl(\omega(\bx_n,\by), \omega(\bx_n,\bz) \Bigr)\, d\mu(\by)d\mu(\bz)\\
     =&\int_{\mathcal M \times \mathcal M}\varphi(d(\bx,\by))\varphi(d(\bx,\bz)) g\Bigl(\omega(\bx,\by), \omega(\bx,\bz) \Bigr)\, d\mu(\by)d\mu(\bz),
 \end{align*}
 as desired. 
 
\textbf{Step 2.} Consider the bilinear map
$$G: L^2([0,R];\rho\mres [0,R];\mathbb R)\times L^2([0,R];\rho\mres [0,R];\mathbb R)\to \mathbb R$$
defined by
$$G(\varphi, \psi):=\mathbb{E}\Bigl[ \varphi(d(\bx,\by))\psi(d(\bx,\bz)) g\Bigl(w(\bx,\by), w(\bx,\bz) \Bigr)\Bigr ].$$
We claim that the map 
 \begin{align}\label{map1}
L^2([0,R];\rho\mres[0,R];\mathbb{R})\rightarrow\mathbb R, \quad  \varphi \mapsto G(\varphi,\varphi):=\mathbb{E}\Bigl[ \varphi(d(\bx,\by))\varphi(d(\bx,\bz)) g\Bigl(w(\bx,\by), w(\bx,\bz) \Bigr)\Bigr ]
    \end{align} is continuous. Using Cauchy–Schwarz inequality
    \begin{align*}
        |G(\varphi,\psi)|\leq \|\varphi\|_{L^2(\rho)}\|\psi\|_{L^2(\rho)}
    \end{align*} 
    Then notice that $$|G(\varphi_n,\varphi_n)-G(\varphi,\varphi)|\leq \|\varphi_n-\varphi\|_{L^2(\rho)}\|\varphi\|_{L^2(\rho)}+ \|\varphi_n-\varphi\|_{L^2(\rho)}\|\varphi_n\|_{L^2(\rho)}.$$ Let $\varphi_n\rightarrow \varphi$ in $\mathcal{H}$ endowed with $L^2(\rho)$ norm, and then we know if $n$ is large enough, we have $\|\varphi_n\|_{L^2(\rho)}\leq 2\|\varphi\|_{L^2(\rho)}$. Consequently, we have 
    $$\lim_{n\rightarrow\infty}G(\varphi_n,\varphi_n)=G(\varphi,\varphi).$$

{\bf Step 3.} Note that \textcolor{black}{by the total law of probability we have}
\begin{align}\label{pconstant2}
        c_{\mathcal H,\mathcal M}&=\inf_{\varphi\in \mathcal H, \|\varphi\|_{L^2(\rho)}=1}  \mathbb E\Bigl[ \varphi(d(\bx,\by))\varphi(d(\bx,\bz)) g\Bigl(w(\bx,\by), w(\bx,\bz) \Bigr)\Bigr ]\\&=\inf_{\varphi\in \mathcal H, \|\varphi\|_{L^2(\rho)}=1}\mathbb{E}_{\bx}\mathbb{E}\Bigl[ \varphi(d(\bx,\by))\varphi(d(\bx,\bz)) g\Bigl(w(\bx,\by), w(\bx,\bz) \Bigr)\Bigr|\bx \Bigr]
    \end{align}
and \textcolor{black}{recall} that 
\begin{align*}
    \mathbb E\Bigl[ \varphi(d(\bx,\by))\varphi(d(\bx,\bz)) g\Bigl(w(\bx,\by), w(\bx,\bz) \Bigr)\Big|\bx\Bigr ]\geq 0
\end{align*}
by the proof of Proposition \ref{mainresult} \textcolor{black}{so that $c_{\mathcal H,\mathcal M}\geq 0$}.

\textcolor{black}{By compactness of $\mathcal{H}$, the continuity of the conditional expectation w.r.t $\bx$  and the continuity of the map \eqref{map1}, one can conclude that $c_{\mathcal H,\mathcal M}$ is positive once the property \eqref{Positvityatx1} is satisfied.} Then using \eqref{C123} in the proof of Proposition \ref{mainresult}, we see that this implies 
    \begin{align*}
           \|A\varphi\|^2_{L^2(\rho)}
     \geq \frac{N-1}{N^2}\|\varphi \|^2_{L^2(\rho)}+c_{\mathcal H, \mathcal M}\frac{N(N-1)(N-2)}{N^3}\|\varphi \|^2_{L^2(\rho)}
    \end{align*}
    which gives the result.
\end{proof}

\begin{remark}\label{remark-2}
Theorem \ref{remark1} provides a way to explore whether we can improve the bound by restricting on a smaller compact subspaces of $L^2([0,R];\rho\mres[0,R];\mathbb{R})$. Using exponential map, one can easily check the condition \eqref{Positvityatx1}. 

 Specifically,
let   $\varphi\in L^2(\rho)$ and $\bx,\by,\bz\sim \mu_0$ i.i.d with densitiy function $p(\cdot).$ We denote $u:=\omega(\bx,\by)$ and $v:=\omega(\bx,\bz)$.
Notice that for $\by\, ,\bz\notin \mathcal{CL}(\bx)$,
\begin{align*}
    u=\exp_{\bx}^{-1}(\by)\quad \textup{and }\quad  v=\exp_{\bx}^{-1}(\bz),
\end{align*}
so $\|u\|=d(\bx,\by)$ and $\|v\|=d(\bx,\bz)$. Moreover, since the volume measure of $\mathcal{CL}(\bx)$ is zero, $\mu_0(\mathcal{CL}(\bx))=0$ too.
In view of \eqref{integral on M}, we obtain that
\begin{align*}
   & \mathbb E_{\by, \bz}\Bigl[\varphi(d(\bx,\by))\varphi(d(\bx,\bz)) g(w(\bx,\by),w(\bx,\bz))\Bigr| \bx\Bigr]\\
   =& \int_{\mathcal M}\int_{\mathcal M}\varphi(d(\bx,\mathbf{y}))\varphi(d(\bx,\mathbf{z}))g_{\bx}\left(w(\bx,\by),w(\bx,\bz)\right)p(\by)p(\bz)\,dV_{\mathcal M}(\by)dV_{\mathcal M}(\bz)\\
   = & \int_{\mathcal D(\bx)}\int_{\mathcal D(\bx)} \varphi(\|u\|)\varphi(\|v\|) g(u,v)p(\exp_{\bx}(u))p(\exp_{\bx}(v))\sqrt{\det (G_{\bx}(u))}\sqrt{\det (G_{\bx}(v))}\, d\lambda^n(u) d\lambda^n(v).
\end{align*}
We choose an orthonormal basis $(\mathbf{e}_i)_{i=1}^n$ at $\bx$ with respect  to the metric $g$ on $T_{\bx}\mathcal M$, so $u=\sum_{i=1}^n u^i \mathbf{e}_i$, $v=\sum_{i=1}^n v^i \mathbf{e}_i$ and
$$g_{\bx}(u,v)=\sum_{i=1}^n u^i v^i.$$
Then
\begin{align*} & \mathbb E_{\by, \bz}\Bigl[\varphi(d(\bx,\by))\varphi(d(\bx,\bz)) g(w(\bx,\by),w(\bx,\bz))\Bigr| \bx\Bigr]\\
=&\sum_{i=1}^n\Bigl[\int_{\mathcal D(\bx)} u^i\varphi(\|u\|) p(\exp_{\bx}(u))\sqrt{\det (G_{\bx}(u))}\,d\lambda^n(u)\Bigr]^2 \geq 0
\end{align*} 
It thus suffices to find $\bx \in \textup{supp}(\mu_0)$ such that\begin{align}\label{easy-condition}
\int_{\mathcal D(\bx)}u^i\varphi(\|u\|)p(\exp_{\bx}(u))\sqrt{\det (G_{\bx}(u))}\,d\lambda^n(u)\neq 0
\end{align} for some $i=1,\dots, n.$
\textcolor{black}{In summary, we simplify the estimation of the condition number of the operator $A$ by leveraging the indistinguishability among particles. This approach narrows the estimation to a functional, denoted by equation \eqref{Positvityatx1}, which only involves a pair of particles relative to a reference particle. Utilizing the coordinate system of the manifold, we further streamline the process to checking  an explicit integral  outlined in \eqref{easy-condition}. The proposed strategy  could be beneficial in choosing hypothesis space wherever it is feasible to compute or estimate the quantity \eqref{easy-condition}. Later, we show it proves effective in both spherical and hyperbolic contexts, as detailed in Propositions \ref{Sphere} and \ref{Hyb}.}

\end{remark}

The reference \cite{Maggioni2021LearningIK} provides numerical examples of first-order opinion dynamics and Lennard-Jones dynamics on both the sphere and the Poincaré disk. These examples motivated our exploration of refining stability results on these spaces, as well as on hyperbolic space more generally.

\subsubsection{Sphere}

We first consider $\mathcal M=\mathbb S^n$, and $\mu=\textup{Unif}((\mathbb S^n)^N)$ is the uniform distribution on sphere.

\begin{proposition}\label{Sphere}
If $\mathcal{M}=\mathbb{S}^n$ and we choose $\mu= \mathrm{Unif}((\mathbb S^n)^{N}).$ Then  one has that
\begin{align*}
    \inf_{\varphi\in \mathcal H, \varphi\neq 0} \frac{\|A\varphi\|_{L^2(\mu)}}{\|\varphi\|_{L^2(\rho)}}=\frac{\sqrt{N-1}}{N}
\end{align*}
for any nontrivial subspace $\mathcal H\subset L^2([0,R];\rho\mres[0,R];\mathbb{R}).$
\end{proposition}
\begin{proof}
Since we consider the uniform distribution, the density function of each particle is $p(\mathbf{u})=\frac{1}{\text{Vol}(\mathbb{S}^n)}.$
Recall that in the sphere case, we have that  
\begin{align*}
&\int_{D(\bx)}\mathbf{u}_i\varphi(\|\mathbf{u}\|)p(\exp_{\mathbf{x}}(\mathbf{u}))\sqrt{G_{\mathbf{x}}(\mathbf{u})}\,d\mathbf{u}\\
    &=\frac{1}{\text{Vol}(\mathbb S^n)}\int_0^\pi {\sin^{n-1}}(t)t\varphi(t)\,dt\int_{\mathbb S^{n-1}}\mathbf u_i\, d\mathcal \sigma(\mathbf{u})\,dt
=0
\end{align*}
for any $i=1,\dots,n,$ as the functions $\mathbf{u}_i$ are odd and domain of integration is symmetric.  The claim then follows from the computation done in Remark \ref{remark-2}.
\end{proof}

 We see that in the sphere case,  \eqref{Positvityatx1} is always zero. As a consequence, 
 on any compact subsets $\mathcal H\subset L^2([0,R];\rho\mres[0,R];\mathbb{R}),$ we find that the $c_{\mathcal H,\mathcal M}$ is equal to zero and consequently, the bound $\frac{\sqrt{N-1}}{N}$ is sharp.

\subsubsection{Hyperbolic space}
Next we show a positive result on the hyperbolic space $\mathbb{H}^n$, where $\mu$ is a product of uniform distribution on a ball of radius $R_0>0$. In this case, we can establish a better bound, depending on $\mathcal H.$
\begin{proposition}\label{Hyb}
 Let $R_0>0,$ $\mathbf{p} \in \mathcal{M}:=\mathbb H^n$ (see definition in section 2), where we assume $ n = 2$ or $n = 3,$ and  $\mu$ have the density function $\frac{1}{\textrm{Vol}_{\mathbb H^n}(B_{R_0}(\mathbf{p}))^N}1_{B_{R_0}(\mathbf{p})\times 
 \cdots \times 
 B_{R_0}(\mathbf{p})}(\textbf X).$ Then for any {\color{black}$\mathcal H\subset \{\varphi\in C([0,2R_0];\rho\mres[0,2R_0];\mathbb R): \supp\varphi\subset [0,R_0), \varphi \text{ only has finite number of zeros in } \supp\varphi\}$} compact (with respect to the topology of $L^2(\rho))$ we get that 
 
\begin{align*}
    \inf_{\varphi\in\mathcal H,\varphi \neq 0} \frac{\|A\varphi\|^2_{L^2(\mu)}}{\|\varphi\|^2_{L^2(\rho)}}\geq \frac{N-1}{N^2}+c_{\mathcal H, \mathcal M}\frac{(N-1)(N-2)}{N^2}.
\end{align*} where $c_{\mathcal H, \mathcal M}>0$ is defined in \eqref{psconstant}. 
\end{proposition}

\begin{proof}
    The density function $p(\bx)$ for marginal distribution of each particle is  given by
\begin{align*}
    p(\bx)=\frac{1}{\text{Vol}_{\mathbb H^n}(B_{R_0}(\mathbf{p}))}1_{B_{R_0}(\mathbf{p})}(\bx)
\end{align*}
and recall that
\begin{align*}
   \mathcal {CL}(\bx) =\emptyset \quad \textup{ for any } \bx \in \mathbb H^n. 
\end{align*}
In view of Equality \eqref{easy-condition}, we only need to consider the integrals
\begin{align}\label{integralsforhyperbolicspace}
    \int_{\mathbb R^n}1_{\exp_{\bx}^{-1}(B_{R_0}(\mathbf{p}))}(\mathbf{u})\varphi(\|\mathbf{u}\|) \mathbf{u}_i \sqrt{G_{\mathbf{x}}(\mathbf{u})}\, d\mathbf{u}\quad \textup{for }i=1,\dots n.
\end{align}
We now pick a geodesic $\gamma$ through $\mathbf{p}$ and let  \textcolor{black}{$q_1,q_2 \in \partial B_{R_0}(\mathbf{p})\cap \gamma(\mathbb R),$ where we assume that $\gamma(0) = q_1$ and $\gamma(R_0) = \mathbf p.$}
 Let $\mathbf{e}_1=\dot \gamma(0)\in T_{\mathbf{q_1}}\mathbb H^n$ and then, completing this into an orthonormal basis $(\mathbf{e}_i),$ gives after parallel translating this along $\gamma$ an orthonormal basis on all of $\gamma.$
 \textcolor{black}
Then the integrals $\eqref{integralsforhyperbolicspace}$ at $\bx = \gamma(t)$ and $i = 1$ are given by 
\begin{align*}
\int_{\exp_{\gamma(t)}^{-1}(B_{R_0}(\mathbf{p}))}\varphi(\|\mathbf{u}\|) \mathbf{u}_1 \frac{\sinh^{n-1}(\|\mathbf u\|)}{\|\mathbf u\|^{n-1}}\, d\mathbf{u}.
\end{align*}
\textcolor{black}
{All we need to show is that this is a non-constant function in $t.$ To do so, we differentiate the integral w.r.t. $t$ and show that for suitable values of $t,$ the derivative is non-zero. Note that the integrand is time independent,
in view of Reynold's transport theorem, see e.g. \cite{leal2007advanced}, we get that
\begin{align}\label{derivative-term}
\frac{d}{dt}\int_{\exp_{\gamma(t)}^{-1}(B_{R_0}(\mathbf{p}))}\varphi(\|\mathbf{u}\|) \mathbf{u}_1 \frac{\sinh^{n-1}(\|\mathbf u\|)}{\|\mathbf{u}\|^{n-1}}\, d\mathbf{u}= \int_{\partial \exp_{\gamma(t)}^{-1}(B_{R_0}(\mathbf{p}))}\varphi(\|\mathbf{u}\|) \mathbf{u}_1 \frac{\sinh^{n-1}(\|\mathbf u\|)}{\|\mathbf{u}\|^{n-1}}{f_t(\mathbf{u})}\, d\sigma,
\end{align}
where $f_t(\mathbf u)=\vec n (\mathbf u, t) \cdot v_t(\mathbf u)$ with $\vec n$ the outward pointing unit normal vector at $\mathbf u  \in \partial \exp_{\gamma(t)}^{-1}(B_{R_0}(\mathbf{p}))$, and 
$v_t$ denotes the velocity vector field on $\partial \exp_{\gamma(t)}^{-1}(B_{R_0}(\mathbf{p}))$ (i.e. the unique vector field on $\partial \exp_{\gamma(t)}^{-1}(B_{R_0}(\mathbf{p}))$ pointing into the future movement of $\partial \exp_{\gamma(t)}^{-1}(B_{R_0}(\mathbf{p}))$ as $t$ varies). Note that since $\exp_{\gamma(t)}^{-1}:\mathbb H^n \rightarrow \mathbb R^n$ is a smooth diffeomorphism (in view of the Hadamard Theorem, see e.g. \cite{DoCarmo}, Chapter 7), we have that 
\begin{align*}
   \partial \exp_{\gamma(t)}^{-1}(B_{R_0}(\mathbf{p})) = \exp_{\gamma(t)}^{-1}(\partial B_{R_0}(\mathbf{p})).
\end{align*}
Moreover, we let $F:[0,\infty) \times \partial B_{R_0}(\mathbf p)\rightarrow \mathbb R^n$ be given by 
\begin{align*}
    F(t,p) = \exp_{\gamma(t)}^{-1}(p),
\end{align*}
so that $F(t, \cdot)$ is a smooth immersion for each $t$. The vector field $v_t: \partial \exp_{\gamma(t)}^{-1}(B_{R_0}(\mathbf{p}))\rightarrow \mathbb R^n$ is then given by 
$$v_t(\mathbf u) = \tfrac{\partial F}{\partial t}(t,\exp_{\gamma(t)}(\mathbf u)).$$}
Since $\varphi \in C_c[0,{R_0})$ and has finite number of zeros in its support,
for any nonzero $\varphi$ we get that there are $0< \varepsilon<T<R_0$ such that $\varphi(t)=0$ for any $t\geq T$ and $\varphi$ is strictly positive or negative on $(T- \varepsilon, T)$.
\textcolor{black}{Now we consider the ball of radius $T$ around $0$ (with respect to the Euclidean distance $\|\cdot \|),$ denoted by $B_T^{\|\cdot \|}.$ 
Let $0<\delta<\varepsilon$ be sufficiently small, then for $t_\delta=2R_0-T+\delta$, $B^{\|\cdot\|}_T\cap \partial \exp_{\gamma(t_\delta)}^{-1}(B_{R_0}(\mathbf p))$ has positive surface measure, moreover
$$T-\varepsilon<\|\mathbf u\|<T\quad \mbox{and}\quad \mathbf u_1>0, \quad \forall \mathbf u\in B^{\|\cdot\|}_T\cap \partial \exp_{\gamma(t_\delta)}^{-1}(B_{R_0}(\mathbf p)).$$
In the meantime, notice that $(T-\delta)\mathbf e_1\in B^{\|\cdot\|}_T\cap \partial \exp_{\gamma(t_\delta)}^{-1}(B_{R_0}(\mathbf p))$ and $v_{t_\delta}((T-\delta)\mathbf e_1)=-\vec n((T-\delta)\mathbf e_1,t_\delta)$, so $f_{t_\delta}((T-\delta)\mathbf e_1)=-1<0$. It follows that for any $\mathbf u\in B^{\|\cdot\|}_T\cap \partial \exp_{\gamma(t_\delta)}^{-1}(B_{R_0}(\mathbf p))$, since $\delta$ is small enough, $f_{t_\delta}(\mathbf u)<0$ by continuity. Note that for $\delta$ sufficiently small
\begin{align*}
\int_{\partial \exp_{\gamma(t_\delta)}^{-1}(B_{R_0}(\mathbf{p}))}\varphi(\|\mathbf{u}\|) \mathbf{u}_1 & \frac{\sinh^{n-1}(\|\mathbf u\|)}{\|\mathbf{u}\|^{n-1}}{f_{t_\delta}(\mathbf{u})}\, d\sigma\\
& =\int_{B^{\|\cdot\|}_T\cap \partial \exp_{\gamma(t_\delta)}^{-1}(B_{R_0}(\mathbf{p}))}\varphi(\|\mathbf{u}\|) \mathbf{u}_1 \frac{\sinh^{n-1}(\|\mathbf u\|)}{\|\mathbf{u}\|^{n-1}}{f_{t_\delta}(\mathbf{u})}\, d\sigma.
\end{align*}
Now the above analysis implies that the integral is non-zero.} This concludes the proof.

\end{proof}

\textcolor{black}{Examples of such hypothesis spaces in Proposition \ref{Hyb} include finite-dimensional piecewise polynomials, trigonometric polynomials, splines restricted to $[0, R_0]$, and real analytic functions on $(0, R_0)$}.

\subsection{Generalization}

Finally, we remark that the stability results also hold on equivalent families of measure $\mu$, which are not necessarily the product of the conditionally i.i.d distribution on $\mathcal{M}^N$. Below we provide an empirical example of learning interaction kernels in opinion dynamics on the sphere. Starting with the initial conditions that follows a distribution with i.i.d components, the distribution $\mu$ for observation data coming from i.i.d trajectories over the time interval [0,5]  does not have i.i.d components anymore. But its associated measure $\rho$ is equivalent to the one formed by the initial condition. So our stability results can be applied to this case as well. 

\begin{figure}[!htbp]
    \centering
    \includegraphics[scale=0.6]{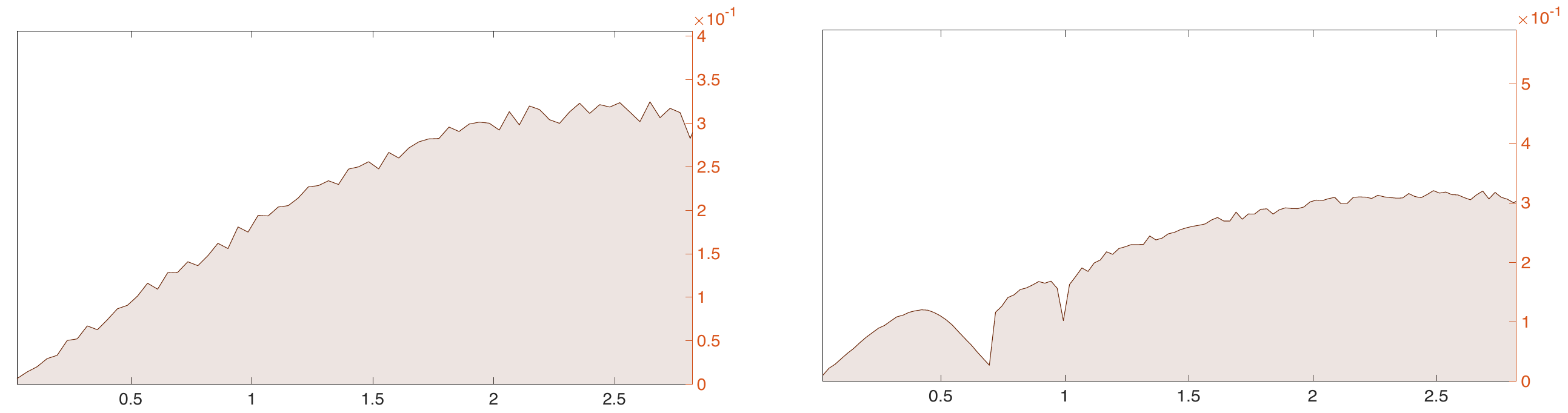}
    \caption{Two examples of $\rho$ from different observation regimes in the same opinion dynamics on $\mathbb{S}^2$ with a piecewise linear kernel $\phi$ that is compactly supported on $[0,\frac{5}{\sqrt{\pi}}]$, where there are 20 agents/opinions evolving as in \cite{Maggioni2021LearningIK} and the injective radius is $\frac{5}{\sqrt{\pi}}$. To obtain empirical approximations of $\rho_1$ and $\rho_2$, we use 3000 trajectories. In the first example, denoted by $\rho_1$, we choose $\mu_1$ to be the product of uniform distribution on $\mathbb{S}^2$ and observe the position and velocity at $t=0$ using i.i.d samples from $\mu_1$. In the second example, denoted by $\rho_2$, we choose $\mu_2$ to be the distribution of positions by observing infinite i.i.d trajectories at time interval [0,5] using $\mu_1$ as the initial distribution. Note that while $\mu_1$ has i.i.d components, $\mu_2$ does not due to time evolution. Despite this difference, we observe numerical evidence that $\rho_2$ is equivalent to $\rho_1$ on $[0,\frac{5}{\sqrt{\pi}}]$. Hence, Corollary \ref{cor} suggests that the stability result also holds for $\rho_2$. }
    \label{fig:my_label_2}
\end{figure}

\begin{corollary}\label{cor}
Suppose that $\tilde u$ and $\tilde \rho$ are equivalent measures to $\mu$ and $\rho\mres[0,R]$ that are defined on $\mathcal{M}^N$ and $[0,R]$ respectively, then $A \in \mathcal{B}(\tilde H, \tilde F)$ where $\tilde H=L^2([0,R];\tilde \rho\mres[0,R];\mathbb{R})$ and $\tilde F=L^2(\mathcal{M}^N;\tilde \mu;T_{\bX}\mathcal{M}^N)$. Moreover, $A$ admits a bounded inverse on $\tilde H$ provided it has a bounded inverse on $\mathcal{H}=L^2(\mathcal{M}^N; \mu;T_{\bX}\mathcal{M}^N)$
\end{corollary}

\begin{proof}
    The conclusion follow by the fact that Radon-Nikodym derivative of two measures are bounded below  and above by two positive constants  on $[0,R]$.
\end{proof}

\section{Examples where the stability result would fail} \label{singularexample}

In the preceding sections, we demonstrated that if $\mu$ is a product of conditionally independent and identically distributed (i.i.d) distributions, then $A$ has a bounded inverse and we derived an estimate on the lower bound of its operator norm. While restricting the space $\mathcal{H}$ to a compact subset of $L^2([0,R];\rho\mres[0,R];\mathbb{R})$ for some manifold can improve this lower bound, it remains sharp for $\mathbb{S}^n$, irrespective of the choice of $\mathcal{H}$. It is worth recalling that we assumed that $\mu$ is absolutely continuous with respect to the volume measure on $\mathcal{M}^N$, and this assumption is fundamental in establishing stability, which was utilized at several key points in the proofs. We shall now provide examples of singular measures where stability may not hold.

Specifically, we start by demonstrating an example of a measure $\mu$ such that $\rhsfo_{\phi}(\bX)=0$ for $\bX \sim \mu$. We choose $\mathcal{M}= \mathbb{S}^n$ for $n\geq 1$ and we pick $N=n+2$ interacting particles. We begin by considering the case where $\mu= \delta_{\bX_0}$  where the vector $\bX_0=[\cdots,\mathbf{v_i},\cdots]^{\top} \in (\mathbb{S}^n)^N$ stacks the vertices of an $n+1$ dimensional  regular unit simplex. In this case, we have
\begin{align*}
\sum_{i=1}^{N}\mathbf{v}_i&=\mathbf{0},\\
|\langle \mathbf{v}_i, \mathbf{v}_j \rangle| &\equiv c_0, i\neq j\\
\omega_{\mathbb{S}^n}(\mathbf{v}_i,\mathbf{v}_j) &= \mathbf{v}_j -\langle \mathbf{v}_j,\mathbf{v}_i\rangle \mathbf{v}_i,
\end{align*} where $c_0$ is a nonzero constant that is the cosine of the angle between two edges. For example, $c_0=\frac{1}{2}$ when $n=1$. From \eqref{distance-on-sphere},  we have $d_{\mathbb{S}^n}(\mathbf{v}_i,\mathbf{v}_j)\equiv \arccos(c_0)$ for all $i\neq j$. Using these expressions,  we have for $i=1,\cdots,N$
\begin{align*}
[\rhsfo_{\phi}(\bX_0)]_i&=\frac{1}{N}\sum_{j\neq i}\phi(\arccos(c_0))\omega(\mathbf{v}_i,\mathbf{v}_j)\\&= \frac{1}{N} \phi(\arccos(c_0))(\sum_{j=1}^N \mathbf{v}_j-\langle \sum_{j=1}^{N} \mathbf{v}_j, \mathbf{v}_i\rangle\mathbf{v}_i)\\
&=0
\end{align*} Therefore, $\rhsfo_{\phi}(\bX)= 0$ whenever $X\sim \mu$.  More generally, let $U$ be any orthogonal matrix, and note that $\rhsfo_{\phi}(U\bX_0)\equiv 0$. If we take $\mu$ to be the uniform distribution over the image of $\bX_0$ under the orthogonal group, then we shall also have $\|A\phi\|_{L^2(\mu)}$=0. However, even in this case,  $\mu$  is still not absolutely continuous with respect to the volume measure on $(\mathbb{S}^n)^N$.

In conclusion, our findings emphasize the importance of the assumptions made in establishing stability results for $\rhsfo_{\phi}(\bX)$, particularly with regards to the absolute continuity of the measure with respect to the volume measure on the underlying manifold. Understanding the impact of singular measures on stability is crucial for the development of more robust and accurate data-driven discovery methods. 

\section{Conclusion}

We investigate the problem of identifiability of interaction kernels in a first-order system from data. We show that for a large family of distributions $\mu$, the statistical inverse problem is well-posed in specific function spaces. In particular, when $\mu$ is also the distribution of independent trajectories with random initialization, our results  provide  positive examples to the geometric coercivity conjecture in \cite{Maggioni2021LearningIK}, which implies that the least square estimators constructed in \cite{Maggioni2021LearningIK} can achieve the mini-max rate of convergence.  In addition, our results indicate that in the mean-field case $(N\rightarrow \infty)$, the inverse problem becomes ill-posed and effective regularization techniques are needed.

\textcolor{black}{The characterization of the stability of the inverse problem, influenced by the underlying geometry, remains largely open. This challenge arises from understanding how the geometry impacts the measure $\mu$ when observing trajectory data, as well as its effect on the integral in Equation \eqref{easy-condition}. We leave these considerations for future research}. There are  also  several other avenues for future work, including investigating identifiability under stronger norms such as the RKHS norm, generalizing to second-order systems, possibly with multiple interaction kernels, and studying mean-field systems on Riemannian manifolds. These directions can potentially lead to a deeper understanding of the identifiability problem and can have important applications in various fields, such as physics, biology, and social sciences.

\section{Appendix}

\begin{definition}\label{positive definite kernel}
Let $\mathbb{X}$ be a nonempty set. A map $k:\mathbb{X}\times \mathbb{X} \rightarrow \mathbb{R}$ is called a positive definite kernel if $k$ is symmetric and
$$\sum_{i,j=1}^{m}c_{i}c_{j}k(x_i,x_j) \geq 0$$ for all $m\in \mathbb{N},\,\{x_1, x_2, \cdots,x_m\}\subset \mathbb{X}$ and $\{c_1,c_2,\cdots, c_m\} \subset \mathbb{R}$.
\end{definition}

 \begin{theorem}[Properties of positive-definite kernels]\label{t52}
 Suppose that $k, k_1, k_2: \mathbb{X}\times\mathbb{X}\to \mathbb{R}$ are positive-definite kernels. Then
\begin{enumerate} \setlength\itemsep{0mm} 
\item $k_1k_2$ is positive-definite.

\item $f(x)f(y)$ is positive-definite for any function $f:\mathbb{X}\to \mathbb{R}$. 

\item $k(g(u),g(v))$ is positive-definite for any  map $g:\mathbb Y\to \mathbb X$ with another nonempty set $\mathbb Y$. 


\item Let $(\mathbb X,\mathcal U)$ be a measurable space with a $\sigma$-finite measure $\mu$. If $k(x,y)$ is measurable and integrable with respect to $\mu\times \mu$, then $\iint k(x,y)d\mu (x)d\mu (y)\ge0$. 
\end{enumerate}
\end{theorem}

\begin{proof}
    Property 1 and 2 can be found in \cite[p.69]{BCR84}. Property 3 can be easily checked by applying the definition \ref{positive definite kernel}. Property 4 can be found in \cite[p. 524, Property 21.2.12]{RKSF}.
\end{proof}

\section*{Acknowledgement}

The authors thank anonymous reviewers for their constructive comments. S. Tang was partially supported by    Faculty Early Career Development Award sponsored by University of California Santa Barbara, Hellman Family Faculty Fellowship  and the NSF DMS-2111303 and NSF DMS-2340631. H. Zhou was partly supported by NSF grant DMS-2109116. M. Tuerkoen wishes to thank Quyuan Lin for helpful comments and discussions.

\bibliography{references}

\newcommand{\etalchar}[1]{$^{#1}$}
\begin{thebibliography}{VDBCR12}

\bibitem[AHPS21]{ahn2022emergent}
Hyunjin Ahn, Seung-Yeal Ha, Hansol Park, and Woojoo Shim.
\newblock Emergent behaviors of cucker--smale flocks on the hyperboloid.
\newblock {\em Journal of Mathematical Physics}, 62(8):082702, 2021.

\bibitem[AHS21]{ahn2021emergent}
Hyunjin Ahn, Seung-Yeal Ha, and Woojoo Shim.
\newblock Emergent dynamics of a thermodynamic cucker-smale ensemble on
  complete riemannian manifolds.
\newblock {\em Kinetic \& Related Models}, 14(2):323, 2021.

\bibitem[AMD17]{OpinionDynamicsongeneralcompactmanifold}
Aylin Aydoğdu, Sean~T. McQuade, and Nastassia~Pouradier Duteil.
\newblock Opinion dynamics on a general compact riemannian manifold.
\newblock {\em Networks \& Heterogeneous Media}, 12(3):489--523, 2017.

\bibitem[B{\etalchar{+}}11]{bishwal2011estimation}
Jaya Prakash~Narayan Bishwal et~al.
\newblock Estimation in interacting diffusions: Continuous and discrete
  sampling.
\newblock {\em Applied Mathematics}, 2(9):1154--1158, 2011.

\bibitem[BFHM17]{bongini2017inferring}
Mattia Bongini, Massimo Fornasier, Markus Hansen, and Mauro Maggioni.
\newblock Inferring interaction rules from observations of evolutive systems i:
  The variational approach.
\newblock {\em Mathematical Models and Methods in Applied Sciences},
  27(05):909--951, 2017.

\bibitem[BPK16]{brunton2016discovering}
Steven~L Brunton, Joshua~L Proctor, and J~Nathan Kutz.
\newblock Discovering governing equations from data by sparse identification of
  nonlinear dynamical systems.
\newblock {\em Proceedings of the national academy of sciences},
  113(15):3932--3937, 2016.

\bibitem[Car92]{DoCarmo}
Manfredo Perdigão~do. Carmo.
\newblock {\em Riemannian geometry / Manfredo do Carmo ; translated by Francis
  Flaherty.}
\newblock Mathematics. Theory and applications. Birkhäuser, Boston, 1992.

\bibitem[Che21]{chen2021maximum}
Xiaohui Chen.
\newblock Maximum likelihood estimation of potential energy in interacting
  particle systems from single-trajectory data.
\newblock {\em Electronic Communications in Probability}, 26:1--13, 2021.

\bibitem[DGYZ22]{du2022discovery}
Qiang Du, Yiqi Gu, Haizhao Yang, and Chao Zhou.
\newblock The discovery of dynamics via linear multistep methods and deep
  learning: Error estimation.
\newblock {\em SIAM Journal on Numerical Analysis}, 60(4):2014--2045, 2022.

\bibitem[DMH22]{della2022lan}
Laetitia Della~Maestra and Marc Hoffmann.
\newblock The lan property for mckean-vlasov models in a mean-field regime.
\newblock {\em arXiv preprint arXiv:2205.05932}, 2022.

\bibitem[DRS20]{duan2020identification}
X~Duan, JE~Rubin, and D~Swigon.
\newblock Identification of affine dynamical systems from a single trajectory.
\newblock {\em Inverse Problems}, 36(8):085004, 2020.

\bibitem[FPP21]{fetecau2021well}
Razvan~C Fetecau, Hansol Park, and Francesco~S Patacchini.
\newblock Well-posedness and asymptotic behavior of an aggregation model with
  intrinsic interactions on sphere and other manifolds.
\newblock {\em Analysis and Applications}, 19(06):965--1017, 2021.

\bibitem[FRT21]{feng2021data}
Jinchao Feng, Yunxiang Ren, and Sui Tang.
\newblock Data-driven discovery of interacting particle systems using gaussian
  processes.
\newblock {\em arXiv preprint arXiv:2106.02735}, 2021.

\bibitem[FZ18]{fetecau2018self}
Razvan~C Fetecau and Beril Zhang.
\newblock Self-organization on riemannian manifolds.
\newblock {\em arXiv preprint arXiv:1802.06089}, 2018.

\bibitem[GCL22]{genon2022inference}
Valentine Genon-Catalot and Catherine Lar{\'e}do.
\newblock Inference for ergodic mckean-vlasov stochastic differential equations
  with polynomial interactions.
\newblock 2022.

\bibitem[GHL04]{GallotRiemannian}
Sylvestre Gallot, Dominique Hulin, and Jacques Lafontaine.
\newblock {\em Riemannian Geometry}.
\newblock Springer, Berlin, Heidelberg, third edition, 2004.

\bibitem[GSW19]{gomes2019parameter}
Susana~N Gomes, Andrew~M Stuart, and Marie-Therese Wolfram.
\newblock Parameter estimation for macroscopic pedestrian dynamics models from
  microscopic data.
\newblock {\em SIAM Journal on Applied Mathematics}, 79(4):1475--1500, 2019.

\bibitem[HKK22]{ha2022emergent}
Seung-Yeal Ha, Myeongju Kang, and Dohyun Kim.
\newblock Emergent behaviors of high-dimensional kuramoto models on stiefel
  manifolds.
\newblock {\em Automatica}, 136:110072, 2022.

\bibitem[HYM{\etalchar{+}}18]{heinonen2018learning}
Markus Heinonen, Cagatay Yildiz, Henrik Mannerstr{\"o}m, Jukka Intosalmi, and
  Harri L{\"a}hdesm{\"a}ki.
\newblock Learning unknown ode models with gaussian processes.
\newblock In {\em International Conference on Machine Learning}, pages
  1959--1968. PMLR, 2018.

\bibitem[Kas90]{kasonga1990maximum}
Raphael~A Kasonga.
\newblock Maximum likelihood theory for large interacting systems.
\newblock {\em SIAM Journal on Applied Mathematics}, 50(3):865--875, 1990.

\bibitem[Lea07]{leal2007advanced}
L~Gary Leal.
\newblock {\em Advanced transport phenomena: fluid mechanics and convective
  transport processes}, volume~7.
\newblock Cambridge university press, 2007.

\bibitem[LLM{\etalchar{+}}21]{li2021identifiability}
Zhongyang Li, Fei Lu, Mauro Maggioni, Sui Tang, and Cheng Zhang.
\newblock On the identifiability of interaction functions in systems of
  interacting particles.
\newblock {\em Stochastic Processes and their Applications}, 132:135--163,
  2021.

\bibitem[LMT21]{lu2021learning}
Fei Lu, Mauro Maggioni, and Sui Tang.
\newblock Learning interaction kernels in stochastic systems of interacting
  particles from multiple trajectories.
\newblock {\em Foundations of Computational Mathematics}, pages 1--55, 2021.

\bibitem[LZTM19]{lu2019nonparametric}
Fei Lu, Ming Zhong, Sui Tang, and Mauro Maggioni.
\newblock Nonparametric inference of interaction laws in systems of agents from
  trajectory data.
\newblock {\em Proceedings of the National Academy of Sciences},
  116(29):14424--14433, 2019.

\bibitem[MB21a]{messenger2021learning}
Daniel~A Messenger and David~M Bortz.
\newblock Learning mean-field equations from particle data using wsindy.
\newblock {\em arXiv preprint arXiv:2110.07756}, 2021.

\bibitem[MB21b]{messenger2021weak}
Daniel~A Messenger and David~M Bortz.
\newblock Weak sindy: Galerkin-based data-driven model selection.
\newblock {\em Multiscale Modeling \& Simulation}, 19(3):1474--1497, 2021.

\bibitem[MMQZ21]{Maggioni2021LearningIK}
Mauro Maggioni, Jason Miller, H.~P.~Truong Qui, and Ming Zhong.
\newblock Learning interaction kernels for agent systems on riemannian
  manifolds.
\newblock {\em International Conference on Machine Learning}, abs/2102.00327,
  2021.

\bibitem[MT14]{motsch2014heterophilious}
Sebastien Motsch and Eitan Tadmor.
\newblock Heterophilious dynamics enhances consensus.
\newblock {\em SIAM review}, 56(4):577--621, 2014.

\bibitem[MTZM20]{miller2020learning}
Jason Miller, Sui Tang, Ming Zhong, and Mauro Maggioni.
\newblock Learning theory for inferring interaction kernels in second-order
  interacting agent systems.
\newblock {\em arXiv preprint arXiv:2010.03729}, 2020.

\bibitem[MXPW11]{miao2011identifiability}
Hongyu Miao, Xiaohua Xia, Alan~S Perelson, and Hulin Wu.
\newblock On identifiability of nonlinear ode models and applications in viral
  dynamics.
\newblock {\em SIAM review}, 53(1):3--39, 2011.

\bibitem[QWX19]{qin2019data}
Tong Qin, Kailiang Wu, and Dongbin Xiu.
\newblock Data driven governing equations approximation using deep neural
  networks.
\newblock {\em Journal of Computational Physics}, 395:620--635, 2019.

\bibitem[RKSF13]{RKSF}
Svetlozar~T Rachev, Lev~B Klebanov, Stoyan~V Stoyanov, and Frank Fabozzi.
\newblock {\em The methods of distances in the theory of probability and
  statistics}, volume~10.
\newblock Springer, 2013.

\bibitem[RPK17]{raissi2017machine}
Maziar Raissi, Paris Perdikaris, and George~Em Karniadakis.
\newblock Machine learning of linear differential equations using gaussian
  processes.
\newblock {\em Journal of Computational Physics}, 348:683--693, 2017.

\bibitem[SKPP21]{sharrock2021parameter}
Louis Sharrock, Nikolas Kantas, Panos Parpas, and Grigorios~A Pavliotis.
\newblock Parameter estimation for the mckean-vlasov stochastic differential
  equation.
\newblock {\em arXiv preprint arXiv:2106.13751}, 2021.

\bibitem[SL09]{schmidt2009distilling}
Michael Schmidt and Hod Lipson.
\newblock Distilling free-form natural laws from experimental data.
\newblock {\em science}, 324(5923):81--85, 2009.

\bibitem[SRS14]{stanhope2014identifiability}
Shelby Stanhope, Jonathan~E Rubin, and David Swigon.
\newblock Identifiability of linear and linear-in-parameters dynamical systems
  from a single trajectory.
\newblock {\em SIAM Journal on Applied Dynamical Systems}, 13(4):1792--1815,
  2014.

\bibitem[Str00]{strogatz2000kuramoto}
Steven~H Strogatz.
\newblock From kuramoto to crawford: exploring the onset of synchronization in
  populations of coupled oscillators.
\newblock {\em Physica D: Nonlinear Phenomena}, 143(1-4):1--20, 2000.

\bibitem[STW18]{schaeffer2018extracting}
Hayden Schaeffer, Giang Tran, and Rachel Ward.
\newblock Extracting sparse high-dimensional dynamics from limited data.
\newblock {\em SIAM Journal on Applied Mathematics}, 78(6):3279--3295, 2018.

\bibitem[VDBCR12]{BCR84}
C~Van Den~Berg, Jens Peter~Reus Christensen, and Paul Ressel.
\newblock {\em Harmonic Analysis on Semigroups: Theory of Positive Definite and
  Related Functions}, volume 100.
\newblock Springer Science \& Business Media, 2012.

\bibitem[YCY22]{yao2022mean}
Rentian Yao, Xiaohui Chen, and Yun Yang.
\newblock Mean-field nonparametric estimation of interacting particle systems.
\newblock {\em arXiv preprint arXiv:2205.07937}, 2022.

\bibitem[YWK21]{yang2021inference}
Shihao Yang, Samuel~WK Wong, and SC~Kou.
\newblock Inference of dynamic systems from noisy and sparse data via
  manifold-constrained gaussian processes.
\newblock {\em Proceedings of the National Academy of Sciences},
  118(15):e2020397118, 2021.

\end{thebibliography}
\bibliographystyle{alpha}

\end{document}